\documentclass[12pt,leqno]{article}
\usepackage{hyperref}
\usepackage{soul}
\usepackage[doc]{optional}
\usepackage{xcolor}
\definecolor{labelkey}{rgb}{0,0.08,0.45}
\definecolor{rekey}{rgb}{0,0.6,0.0}
\definecolor{Brown}{rgb}{0.45,0.0,0.05}
\usepackage{exscale,relsize}
\usepackage{amsmath}
\usepackage{amsfonts}
\usepackage{amssymb}
\usepackage{calc}
\usepackage{theorem}
\usepackage{pifont}      
\usepackage{graphicx}
\DeclareMathOperator{\weakly}{\rightharpoonup^{\mathrm{w}}}

\newcommand{\w}{\ensuremath{\operatorname{w}}}

\newcommand{\scal}[2]{\langle{{#1},{#2}}\rangle}

\newcommand{\RR}{\ensuremath{\mathbb R}}

\newcommand{\RX}{\ensuremath{\,\left]-\infty,+\infty\right]}}

\newcommand{\NN}{\ensuremath{\mathbb N}}

\newcommand{\menge}[2]{\big\{{#1} \mid {#2}\big\}}

\newcommand{\To}{\ensuremath{\rightrightarrows}}

\newcommand{\dom}{\ensuremath{\operatorname{dom}}}

\newcommand{\gra}{\ensuremath{\operatorname{gra}}}
\newcommand{\epi}{\ensuremath{\operatorname{epi}}}

\newcommand{\inte}{\ensuremath{\operatorname{int}}}

\newcommand{\ran}{\ensuremath{\operatorname{ran}}}

\newcommand{\Id}{\ensuremath{\operatorname{Id}}}

\newcommand{\tr}{\ensuremath{\mathbf{tr}}}

\newcommand{\argmin}{\ensuremath{\operatorname{argmin}}}

\newtheorem{theorem}{Theorem}[section]
\newtheorem{lemma}[theorem]{Lemma}
\newtheorem{fact}[theorem]{Fact}
\newtheorem{corollary}[theorem]{Corollary}
\newtheorem{proposition}[theorem]{Proposition}
\newtheorem{definition}[theorem]{Definition}

\theoremstyle{plain}{\theorembodyfont{\rmfamily}
}
\theoremstyle{plain}{\theorembodyfont{\rmfamily}
}
\theoremstyle{plain}{\theorembodyfont{\rmfamily}
}
\theoremstyle{plain}{\theorembodyfont{\rmfamily}
\newtheorem{example}[theorem]{Example}}
\theoremstyle{plain}{\theorembodyfont{\rmfamily}
\newtheorem{remark}[theorem]{Remark}}

\theoremstyle{plain}{\theorembodyfont{\rmfamily}
}

\newcommand{\qede}{\hspace*{\fill}$\Diamond$\medskip}
\begin{document}


\title{\sffamily{Legendre-type integrands and convex integral functions}}

\author{
Jonathan M. Borwein\thanks{CARMA, University of Newcastle,
 Newcastle, New South Wales 2308, Australia. E-mail:
\texttt{jonathan.borwein@newcastle.edu.au}. Distinguished Professor King Abdulaziz University, Jeddah.}\;
  and Liangjin\
Yao\thanks{CARMA, University of Newcastle,
 Newcastle, New South Wales 2308, Australia.
E-mail:  \texttt{liangjin.yao@newcastle.edu.au}.}}

\date{August 24,  2012}
\maketitle

\begin{abstract} \noindent
In this paper, we study the properties of integral functionals induced on $L^1_E (S,\mu)$ by closed convex functions on a Euclidean space $E$.
We give  sufficient conditions  for such integral functions to be strongly rotund (well-posed). We show that
in this generality functions such as the Boltzmann-Shannon entropy and the Fermi-Dirac entropy  are strongly rotund. We also study convergence in measure and give various limiting counterexample.
\end{abstract}



\section{Introduction}

We assume throughout that
$X$ is a real Banach space with norm $\|\cdot\|$,
that $X^*$ is the continuous dual of $X$,
 and
that $X$ and $X^*$ are paired by $\scal{\cdot}{\cdot}$.
The \emph{open  unit ball} and the \emph{closed unit ball} in $X$ is denoted
respectively
by $U_X :=\{x\mid \|x\|<1\}$ and
$B_X:=
\menge{x\in X}{\|x\|\leq1}$, $U(x,\delta):=x+\delta U_X$ and
$B(x,\delta):=x+\delta B_X$ (where $\delta\geq0$ and $x\in X$) and $\NN=\{1,2,3,\ldots\}$.
 We also assume that $d\in\NN$ and reserve $E$ for
a Euclidean space $\RR^d$ with the induced norm $\|\cdot\|$.

Throughout the  paper, we also  assume that
 $S$ is an arbitrary non-trivial set and that $(S,\mu)$ is a complete \emph{finite} measure space (with
nonzero measure $\mu$).
The Banach space  $L^1_E (S,\mu)$
with $\|\cdot\|_1$ stands for the  space of
 all (equivalence classes of) measurable
functions $f: S\rightarrow\RR^n$ such that $\int_S\|f(s)\|{\rm
d}\mu(s)< +\infty$. The norm $\|\cdot\|_1$ on $L^1_E (S,\mu)$ and
$\langle\cdot,\cdot\rangle $ on $L^1_E (S,\mu)\times \big(L^1_E
(S,\mu)\big)^*\big(=L^{\infty}_E (S,\mu)\big)$ are respectively
defined by
\begin{align*}
\|f\|_1:=\int_S\|f(s)\|{\rm d}\mu(s)\quad\text{and} \quad\langle
f,g\rangle:= \int_S\langle f(s), g(s)\rangle {\rm d}\mu(s),\quad
\forall f\in L^1_E (S,\mu), g\in L^{\infty}_E (S,\mu).
\end{align*}
The norm on $L^{\infty}_E (S,\mu)$ is $\|\cdot\|_{\infty}$.

Let $A\colon X\To X^*$
be a \emph{set-valued operator} (also known as a relation,
 point-to-set mapping or multifunction)
from $X$ to $X^*$, i.e., for every $x\in X$, $Ax\subseteq X^*$,
and let
$\gra A := \menge{(x,x^*)\in X\times X^*}{x^*\in Ax}$ be
the \emph{graph} of $A$. The \emph{domain} of $A$
  is $\dom A:= \menge{x\in X}{Ax\neq\varnothing}$ and
$\ran A:=A(X)$ is the \emph{range} of $A$.

Recall that $A$ is
\emph{monotone} if
\begin{equation}
\scal{x-y}{x^*-y^*}\geq 0,\quad \forall (x,x^*)\in \gra A\;
\forall (y,y^*)\in\gra A,
\end{equation}
and \emph{maximally monotone} if $A$ is monotone and $A$ has
 no proper monotone extension
(in the sense of graph inclusion).

We now recall  some additional standard notations \cite{BorVan}.
We denote by $\longrightarrow$
and $\weakly$ respectively,
the norm convergence and weak convergence
of  sequences.
Given a subset $C$ of $X$,
$\inte C$ is the \emph{interior} of $C$ and
$\overline{C}$ is   the
\emph{norm closure} of $C$.
Let $(C_n)_{n\in\NN}$ be  a sequence of subsets in $X$. We define $\overline{\lim}^{\w}C_n$ by
$\overline{\lim}^{\w}C_n:=\big\{x\in X\mid
 \exists x_{n_k}\in C_{n_k}\,\text{with}\, x_{n_k}\weakly x\big\}$.
Let $f\colon X\to \RX$ and $\lambda\in\RR$. Then
$\dom f:= f^{-1}(\RR)$ is the \emph{domain} of $f$.
We say $f$ is proper if $\dom f\neq\varnothing$.
The \emph{lower level sets} of $f$ are
the sets  $\{x\in X\mid f(x)\leq \lambda\}$.
The \emph{epigraph} of $f$ is $\epi f := \menge{(x,r)\in
X\times\RR}{f(x)\leq r}$.
Let $C$ be convex, we say $x\in C$ is an \emph{extreme point} of $C$
if $\lambda u+(1-\lambda)v\neq x,
\forall u, v\in C\backslash\{x\}, \forall \lambda\in\left[0,1\right]$.
If $x\in\argmin f$, then $f(x)=\inf\{f(y)\mid y\in X\}$.
Let $f$ be proper. The \emph{subdifferential} of
$f$ is defined by
   \begin{align*}\partial f\colon X\To X^*\colon
   x\mapsto \big\{x^*\in X^*\mid(\forall y\in
X)\; \scal{y-x}{x^*} + f(x)\leq f(y)\big\}.
\end{align*}
We say $f$ has the \emph{Kadec} or \emph{Kadec--Klee} property if the following implication
\begin{align*}x_n\weakly x\in\dom  f, f(x_n)\longrightarrow f(x)\quad \Rightarrow\quad x_n\longrightarrow x.
\end{align*}
holds.

As in \cite{BorLew1} we say that  $f$  is \emph{strongly rotund} if $f$ is strictly convex on its domain, $f$ has weakly compact lower level sets, and $f$ has the Kadec property. This is in effect a well-posedness condition, see \cite{Luc}.

Let $\phi:E\rightarrow\RX$ be proper lower semicontinuous and convex.
We define $I_{\phi}: L^1_E (S,\mu)\rightarrow\RX$ by
\begin{align*}
x\mapsto\int_S \phi (x(s)) {\rm d}\mu(s).
\end{align*}

The integral function $I_{\phi}$ has attracted much interest, see, e.g., \cite{Rock68PJM, Rock71PJM,Rock7604,Vin,BorLew3,BorLew2,BorVan, TeVa, BorVan2,CsFr} and the references given therein. In the one-dimensional case with Lebesgue measure, Borwein and Lewis  presented some characterizations for the integral  function $I_{\phi}$ to be strongly rotund (See \cite{BorLew1}.).
 In this paper, we extend their work to an arbitrary Euclidean space.

\subsection{Organization of the paper}
The remainder of this paper is organized as follows.
In Section~\ref{s:aux}, we collect preliminary results for future reference
and  the
reader's convenience. In Section~\ref{s:voi}, we present a sufficient condition for the integral  function $I_{\phi}$ to be strongly rotund in our main result (Theorem~\ref{Tsge2:1}). Some examples and applications are provided in Section~\ref{s:ExA}, in which we show that
the Boltzmann-Shannon entropy and the Fermi-Dirac entropy defined on
$L^1_E (S,\mu)$ both are strongly rotund.   In Section \ref{sec:burg} we present an enlightening illustration of failure of strong rotundity.
In Section~\ref{ApVisin}, we apply a lovely  result due to Visintin to both strengthen Theorem~\ref{Tsge2:1} and to shed light on the Kadec property. In the final Section~\ref{sec:meas} we  turn to the role of convergence in measure.

\section{Preliminary results}\label{s:aux}
We first introduce Vitali's covering theorem.

\begin{fact}[Vitali]\label{VitaliCo}\emph{(See \cite[Theorem~1, page~27]{EvanGar}.)}
 Let $(x_i)_{i\in I}$ be in $E$ and  $(\delta_i)_{i\in I}$ be in $\left]0,+\infty\right[$
  such that $\sup_{i\in I}\delta_i<+\infty$.
  Then there exists a countable subset $\Gamma$ of $I$ such that $B(x_{\alpha},\delta_{\alpha})\cap B(x_{\beta},\delta_{\beta})=\varnothing$ (for every $\alpha,\beta\in\Gamma$ with $\alpha\neq\beta$) and
  \begin{align*}
  \bigcup_{i\in I} B(x_i,\delta_i)\subseteq \bigcup_{i\in\Gamma} B(x_i,5\delta_i).
 \end{align*}

 \end{fact}

\begin{corollary}\label{CoVitaliCo}
 Let $U$ be an open subset of $E$ and $\delta>0$. Then there exist a sequence $(x_n)_{n\in\NN}$ in $U$ and a sequence  $(\delta_n)_{n\in\NN}$ in $\left]0,\delta\right]$
  such that $B(x_{n},\frac{\delta_{n}}{5})\cap B(x_{m},\tfrac{\delta_{m}}{5})=\varnothing$ (for every $n,m\in\NN$ with $n\neq m$) and
 \begin{align}
\bigcup_{n\in\NN} B(x_n,\delta_n)=U.\label{CoVitaliCo:Ea1}
 \end{align}

\end{corollary}
\begin{proof}
Let $x\in U$. There exists $\beta_x \in\left]0,\tfrac{\delta}{5}\right]$ such that
\begin{align}B(x,5\beta_x)\subseteq U.\label{CoVitaliCo:E1}\end{align}
Then we have
\begin{align}
U\subseteq \bigcup_{x\in U} B(x,\beta_x).
\end{align}
By Fact~\ref{VitaliCo}, there exist a countable set $I$ and $(x_i)_{i\in I}$ in $U$ such that $B(x_i,\beta_{x_i})\cap B(x_j,\beta_{x_j})=\varnothing$ (for every $i,j\in I$ with $i\neq j$) and
\begin{align}
U\subseteq \bigcup_{x\in U} B(x,\beta_x)\subseteq \bigcup_{i\in I} B(x_i,5\beta_{x_i}).\label{CoVitaliCo:E2}
\end{align}
Then by \eqref{CoVitaliCo:E1}, $\bigcup_{i\in I} B(x_i,5\beta_{x_i})\subseteq U$. Hence by \eqref{CoVitaliCo:E2},
\begin{align}
U= \bigcup_{i\in I} B(x_i,5\beta_{x_i}).\label{CoVitaliCo:E3}
\end{align}
Note that $I$ cannot be a finite set. Otherwise, $\bigcup_{i\in I} B(x_i,5\beta_{x_i})$ is closed, which contradicts \eqref{CoVitaliCo:E3}.
Set $\alpha_i :=5\beta_{x_i}, \forall i\in I$.
Thus \eqref{CoVitaliCo:E3} implies that  \eqref{CoVitaliCo:Ea1} holds.
\end{proof}

\begin{fact}[Dunford]\label{Dundcom:1}\emph{(See \cite[Theorem~4, page~104]{DiestelJr}.)}
 Let $D$ be a weakly compact subset of $L^1_E(S,\mu)$. Then for every $\varepsilon>0$,
 there exists $\delta>0$ such that
  \begin{align*}
 \int_C\big\|y(s)\big\| {\rm d}\mu(s)\leq\varepsilon,\quad\forall \mu(C)\leq\delta,\,\forall y\in D.
 \end{align*}

 \end{fact}

\begin{fact}[Rockafellar]\emph{(See \cite[Theorem~1]{Rock69} or  \cite[Theorem~2.28]{ph}.)}
\label{pheps:11}Let $A:X\To X^*$ be  monotone with
 $\inte\dom A\neq\varnothing$.
Then $A$ is locally bounded at $x\in\inte\dom A$,
that is, there exist $\delta>0$ and $K>0$ such that
\begin{align*}\sup_{y^*\in Ay}\|y^*\|\leq K,
\quad \forall y\in (x+\delta B_X)\cap \dom A.
\end{align*}
\end{fact}

\begin{fact}\emph{(See  \cite[Theorem~2.2.1]{Zalinescu}.)}
\label{LoweSC:1}Let $f:X\rightarrow \RX$ be  proper convex. Then $f$ is lower semicontinuous
if and only if $f$ is weak--lower semicontinuous.
\end{fact}

\begin{fact}[Borwein and Lewis]\emph{(See  \cite[Lemma~2.8]{BorLew1}.)}
\label{BorLewNorm:1}Let $f:X\rightarrow \RX$ be  proper lower semicontinuous and convex. Suppose that $f^*$ is Fr\'{e}chet differentiable on $\dom\partial f^*$. Assume that $(x_n)_{n\in\NN}$ and $x\in\dom \partial f$ are  such that
$x_n\weakly x, f(x_n)\longrightarrow f(x)$. Then $x_n\longrightarrow x$.
\end{fact}

\begin{definition}\emph{(See \cite{BBC1}.)}
Let $f:X\rightarrow\RX$ be  proper  lower semicontinuous and convex. We say
\begin{enumerate}
 \item $f$ is \emph{essentially smooth}
 if $\partial f$ is locally bounded and single-valued on its
domain.
\item  $f$ is \emph{essentially strictly convex}
if $(\partial f)^{-1}$ is locally bounded on its domain
and $f$ is strictly convex on every convex subset of $\dom\partial f$.
\item $f$ is \emph{Legendre} if $f$ is essentially smooth
and essentially strictly convex.

 \end{enumerate}
\end{definition}

\begin{fact}[Rockafellar]\label{phiEmth}
\emph{(See \cite[Theorem~26.3]{Rock70CA}.)}
Let $\phi:E\rightarrow\RX$ be proper lower semicontinuous and convex.
 Then ${\phi}$ is
essentially strictly convex if and only if $\phi^*$ is essentially smooth. \qede
\end{fact}

\begin{fact}\emph{(See \cite[Theorem~5.6(ii)\&(iii) and Theorem~5.11(ii)]{BBC1}.)}
\label{legenFa:1}
Let $f:X\rightarrow\RX$ be  proper  lower semicontinuous and convex. Then the following hold.
\begin{enumerate}
 \item \label{legenFa:1E1}$f$ is essentially smooth if and only if
 $\inte\dom f\neq\varnothing$ and $\partial f$ is single-valued,
if and only if
 $\inte\dom f=\dom\partial f$ and $\partial f$ is single-valued.

\item \label{legenFa:1E3} Suppose that $X=E$.
Then $f$ is essentially strictly convex if and only if
$f$ is strictly convex on every convex subset of $\dom\partial f$.
 \end{enumerate}
\end{fact}

\begin{fact}\label{ExVis:1}\emph{(See \cite[Fact~5.3.3, page~239]{BorVan}.)}
Let $f:X\rightarrow\RX$ be  proper  lower semicontinuous and strictly convex. Then $(x, f(x))$ is an extreme point of
$\epi f$ for every $x\in\dom f$.
\end{fact}

\begin{lemma}\label{Ecbness}
Let $A:E\To E$ be  monotone with
 $\inte\dom A\neq\varnothing$.   Let $C$ be a bounded closed subset of
 $\inte\dom A$. Then there exists $M>0$ such that
 \begin{align*}
 \sup_{a^*\in Aa,\, a\in C}\|a^*\|\leq M.
 \end{align*}

\end{lemma}
\begin{proof}
Let $x\in C$. By Fact~\ref{pheps:11}, there exist
$\delta_x >0$ and $M_x >0$ such that
\begin{align}
\sup_{a^*\in Aa, \, a\in U(x,\delta_x)}\|a^*\|\leq M_x.\label{Ecbness:Ea1}
\end{align}
Then we have
\begin{align}
C\subseteq\bigcup_{x\in C}U(x,\delta_x).
\end{align}
Since $C$ is compact, there exists $N\in\NN$ such that
$(x_n)_{n=1}^{N}$ in $C$ and
\begin{align}
C\subseteq\bigcup^{N}_{n=1}U(x_n,\delta_{x_n}).\label{Ecbness:E1}
\end{align}
Set $M:=\max\{M_{x_n}\mid n=1,\cdots, N\}$.
Then by \eqref{Ecbness:Ea1} and \eqref{Ecbness:E1},
$
\sup_{a^*\in Aa,\, a\in C}\|a^*\| \leq M$.
\end{proof}

\begin{remark}  If $C$ is assumed norm compact, this proof remains valid in a general Banach space. \qede \end{remark}

In the following subsection we turn to  properties of the function $I_{\phi}$.

\subsection{Basic properties of $I_\phi$}\label{ss:I}

\begin{fact}[Rockafellar]\label{Intephi}
\emph{(See 
\cite[Theorem 3C and Theorem~3H]{Rock7604} and \cite[Exercise~6.3.7, page~306]{BorVan}.)}
Let $\phi:E\rightarrow\RX$ be proper, lower semicontinuous and convex.
Then
$I_{\phi}$ is proper lower semicontinuous and convex, and $
I^*_{\phi}=I_{\phi^*}$.
Moreover,
\begin{align*}
x^*\in\partial I_{\phi}(x)\Longleftrightarrow
\big( x^*(s)\in\partial \phi(x(s))\,\text{for almost all\ $s$\ in\ $S$}\big).
\end{align*}

\end{fact}
\begin{remark}\label{reInteRoc}Let $\phi:E\rightarrow\RX$ be proper lower semicontinuous and convex.
By Fact~\ref{LoweSC:1} and Fact~\ref{Intephi}, $I_{\phi}$ is proper weak--lower semicontinuous and convex. \qede
\end{remark}

The following three results were proved by Borwein and Lewis when $E=\RR$.
Their proofs can be adapted to  the general space $E$. For the readers' convenience, we record full proofs herein.
\begin{fact}\emph{(See \cite[Lemma~3.1)]{BorLew1}.)}
\label{legenFa:2}Let $\phi:E\rightarrow\RX$ be proper, lower semicontinuous and convex.
Then
$I_{\phi}$ is strictly convex on its domain if and only if $\phi$
is strictly convex on its domain.
\end{fact}
\allowdisplaybreaks
\begin{proof}
``$\Rightarrow$":
 Let $ v, w\in \dom\phi$ with $v\neq w$. Set  $x(s):=v$ and $y(s):=w$ for every $s\in S$.
Then $\{x,y\}\subseteq L^1_E(S,\mu)$ and  $x\neq y$. Let $\lambda\in\left]0,1\right[$.
Since $I_{\phi}$ is strictly convex on its domain,
\begin{align*}
&\phi\big(\lambda v+(1-\lambda)w\big)=\frac{1}{\mu(S)}\int_S \phi\big(\lambda v+(1-\lambda)w\big){\rm d}\mu(s)\\
&=\frac{1}{\mu(S)}\int_S \phi\big(\lambda x(s)+(1-\lambda)y(s)\big){\rm d}\mu(s)\\
&=\frac{1}{\mu(S)}I_{\phi}\big(\lambda x+(1-\lambda)y\big)\\
&<\frac{1}{\mu(S)}\lambda I_{\phi}(x)+\frac{1}{\mu(S)}(1-\lambda) I_{\phi}(y)\\
&=\frac{1}{\mu(S)}\lambda \int_S {\phi}(v)d \mu(s)+\frac{1}{\mu(S)}(1-\lambda) \int_S {\phi}(w){\rm d}\mu(s)\\
&=\lambda {\phi}(v)+(1-\lambda) {\phi}(w).
\end{align*}
Hence $\phi$ is strictly  convex on its domain.

``$\Leftarrow$":  By Fact~\ref{Intephi}, $I_\phi$ is convex. Suppose to the contrary that
$I_{\phi}$ is not strictly convex on its domain. Then there exists $\lambda\in\left]0,1\right[$ and
$\{x,y\}\subseteq \dom I_{\phi}$ with  $x\neq y$ such that
\begin{align*}
I_{\phi}\big(\lambda x+(1-\lambda)y\big)-\lambda I_{\phi}(x)-(1-\lambda) I_{\phi}(y)=0.
\end{align*}
Then we have
\begin{align}\label{legenFa:2E1}
\int_S\Big(\lambda {\phi}(x(s))+(1-\lambda)
{\phi}(y(s))-\phi\big(\lambda x(s)+(1-\lambda)y(s)\big)\Big){\rm
d}\mu(s)=0.
\end{align}
Since $\phi$ is convex, \begin{align}g(s):=\lambda
{\phi}(x(s))+(1-\lambda) {\phi}(y(s))-\phi\big(\lambda
x(s)+(1-\lambda)y(s)\big)\geq0.\end{align} Set $T_m:=\big\{s\in
S\mid g(s)\geq \tfrac{1}{m}\big\},\,\forall m\in\NN$. Then by
\eqref{legenFa:2E1}, we have \begin{align*}0=\int_S g(s){\rm
d}\mu(s)\geq \int_{T_m} g(s){\rm d}\mu(s)\geq \tfrac{1}{m}\mu
(T_m),\quad\forall m\in\NN.\end{align*} Hence $\mu (T_m)=0, \forall
m\in\NN$. Thus, $\phi\big(\lambda x(s)+(1-\lambda)y(s)\big)-\lambda
{\phi}(x(s))-(1-\lambda) {\phi}(y(s))=0$ for all almost $s\in S$.
Since $\phi$ is strictly convex its domain, $x(s)=y(s)$ for all
almost $s\in S$. Hence $x$ is equivalent to $y$ and thus $x=y$,
which contradicts that $x\neq y$.
\end{proof}

Following  \cite{BorLew1}, given a measurable set $T\subseteq S$  we denote by $T^c:=\big\{s\in S\mid s\notin T\}$ and we denote the restriction of $\mu$ and $x\in L^1_E (S,\mu)$ to $T$ respectively by
 $\mu|_T$ and $x|_T$. We define $I_{\phi}^T: L^1_E (T,\mu |_T)\rightarrow\RX$ by
 $I_{\phi}^T(z):=\int_T\phi(z(s)){\rm d}\mu(s)$.

\begin{fact}\emph{(See \cite[Lemma~3.5)]{BorLew1}.)}
\label{legenFa:3}Let $\phi:E\rightarrow\RX$ be proper, lower semicontinuous and convex, and $T$ be a measurable subset of $S$.
Suppose that $(x_n)_{n\in\NN}$ in $L^1_E (S,\mu)$ such that $x_n \weakly x$. Then
$x_n|_T\weakly  x|_T$ in $L^1_E (T,\mu |_T)$. Moreover, if
$I_{\phi}(x_n)\longrightarrow I_{\phi}(x)<+\infty$, then
$I_{\phi}^T( x_n|_T)\longrightarrow I_{\phi}^T( x|_T)<+\infty$.
\end{fact}
\begin{proof}
We first show that $x_n|_T\weakly  x|_T$.
Let $x^*\in L^{\infty}_E(T,\mu)$. Then we define $y^*$ by $y^*(s):= x^*(s)$, if $s\in T$;
$y^*(s):= 0$, if $s\in T^c$. Then $y^*\in L^{\infty}_E(S,\mu)$ and
$\langle x_n|_T, x^*\rangle=\langle x_n, y^*\rangle\longrightarrow
\langle x, y^*\rangle=\langle x|_T, x^*\rangle$.
Hence $x_n|_T\weakly  x|_T$.

Now we show that $I_{\phi}^T( x_n|_T)\longrightarrow I_{\phi}^T( x|_T)<+\infty$.
Since $x_n|_T\weakly  x|_T$ and $x_n|_{T^c}\weakly  x|_{T^c}$,  by Fact~\ref{Intephi} and Remark~\ref{reInteRoc},
\begin{align}\liminf I_{\phi}^T( x_n|_T)\geq I_{\phi}^T( x|_T)\quad\text{and}\quad
\liminf I_{\phi}^{T^c}( x_n|_{T^c})\geq I_{\phi}^{T^c}( x|_{T^c}). \label{legenFa:3E1}
\end{align}
Then we have
\begin{align*}
&\limsup I_{\phi}^T( x_n|_T)=\limsup \big(I_{\phi}( x_n)-I_{\phi}^{T^c}( x_n|_{T^c})\big)\\
&=\lim I_{\phi}( x_n)-\liminf I_{\phi}^{T^c}( x_n|_{T^c})\\
&\leq I_{\phi}(x)-I_{\phi}^{T^c}( x|_{T^c})\\
&=I_{\phi}^{T}( x|_{T})<+\infty\quad\text{(since $I_{\phi}(x)<+\infty$ and $I_{\phi}^{T^c}( x|_{T^c})>-\infty$ by Fact~\ref{Intephi})}.
\end{align*}
Then by \eqref{legenFa:3E1}, $\lim I_{\phi}^T( x_n|_T)=I_{\phi}^{T}( x|_{T})$.
\end{proof}

\begin{fact}\emph{(See \cite[Lemma~3.2]{BorLew1}.)}\label{BLal:1}
Let $\phi:E\rightarrow\RX$ be proper, lower semicontinuous and convex. Then
$I_{\phi^*}$ is Fr\'{e}chet differentiable everywhere on $L_E^{\infty}(S, \mu)$
if and only if $\phi^*$ is differentiable everywhere on $E$.
\end{fact}

\begin{proof}
``$\Rightarrow$": Let $z\in E$ and set  $w(s):=z$ for every $s\in S$. Then $w(s)\in L^{\infty}_E(S,\mu)$.
Then we have $\phi^* (z)\mu(S)=I_{\phi^*} (w)<+\infty$ and hence $z\in\dom\phi^*$.
Thus ${\phi^*}$ is full domain.
Let $u,v\in E$. Now we show $\phi^*$ is differentiable at $u$.  Set $x(s):=u$ and $y(s):=v, \forall s\in S$.
Then $\{x(s), y(s)\}\subseteq L^{\infty}_E(S,\mu)$.
Let $t>0$. Then we have
\allowdisplaybreaks
 \begin{align*}
&\frac{\phi^*(u+tv)+\phi^*(u-tv)-2\phi^*(u)}{t}\\
&=\tfrac{1}{\mu(S)}\int_S\frac{\phi^*(u+tv)+\phi^*(u-tv)-2\phi^*(u)}{t}{\rm d}\mu(s)\\
 &=\tfrac{1}{\mu(S)}\int_S\frac{\phi^*\big(x(s)+ty(s)\big)+\phi^*(x(s)-ty(s))-2\phi^*(x(s))}{t}{\rm d}\mu(s)\\
 &=\tfrac{1}{\mu(S)}\frac{I_{\phi^*}\big(x+ty\big)+I_{\phi^*}(x-ty)-2I_{\phi^*}(x)}{t}\longrightarrow 0\quad\text{as}\quad
 t\longrightarrow0\quad\text{(by \cite[Exercise~1.24]{ph})}.
 \end{align*}
By \cite[Exercise~1.24]{ph} again, $\phi^*$ is differentiable at $u$.

``$\Leftarrow$": By \cite[Corollary, page~20]{ph},
$(\phi^*)'$ is continuous on $E$. Let $x^*\in L_E^{\infty}(S, \mu)$.
Then there exists $M>0$ such that $\|x^*(s)\|\leq M$ almost everywhere. We can and do suppose that
$\|x^*(s)\|\leq M, \forall s\in S$.
Since $M B_{E}$ is compact, $(\phi^*)'$ is uniformly continuous on $M B_{E}$.
Let $\varepsilon>0$. There exists $\delta>0$ such that
\begin{align}\big\|(\phi^*)' (x^*(s)+v)-(\phi^*)' x^*(s)\big\|\leq
\frac{\varepsilon}{\mu(S)}, \quad\forall\|v\|\leq\delta.\label{BLal:1E1}
\end{align}
Let $y^*\in L_E^{\infty}(S, \mu)$ with $\|y^*\|_{\infty}\leq\delta$. Then applying Mean Value Theorem,  we have
\begin{align*}
&\Big\|\int_S \phi^* \big(x^*(s)+y^*(s)\big){\rm d}\mu(s)-
\int_S \phi^* (x^*(s)){\rm d}\mu(s)-\int_S \Big\langle(\phi^*)'(x^*(s)),y^*(s)\Big\rangle {\rm d}\mu(s)\Big\|\\
&=\Big\|\int_S \left[\phi^* \big(x^*(s)+y^*(s)\big)-
\phi^* (x^*(s))- \Big\langle(\phi^*)'(x^*(s)),y^*(s)\Big\rangle\right]{\rm d}\mu(s)\Big\|\\
&=\Big\|\int_S \left[\Big\langle(\phi^*)' \big(x^*(s)+t_s y^*(s)\big), y^*(s)\Big\rangle- \Big\langle(\phi^*)'(x^*(s)),y^*(s)\Big\rangle\right]{\rm d}\mu(s)\Big\|,\quad \exists t_s\in\left]0,1\right[\\
&=\Big\|\int_S \left[\Big\langle(\phi^*)' \big(x^*(s)+t_s y^*(s)\big)-
(\phi^*)'(x^*(s)),y^*(s)\Big\rangle\right]{\rm d}\mu(s)\Big\|,\quad \exists t_s\in\left]0,1\right[\\
&\leq \int_S \Big\|\Big\langle(\phi^*)' \big(x^*(s)+t_s y^*(s)\big)-
(\phi^*)'(x^*(s)),y^*(s)\Big\rangle\Big\|{\rm d}\mu(s),\quad \exists t_s\in\left]0,1\right[\\
&\leq \int_S \frac{\varepsilon}{\mu(S)}\|y^*(s)\|{\rm d}\mu(s)\quad\text{(by \eqref{BLal:1E1})}\\
&\leq\frac{\varepsilon}{\mu(S)}\|y^*\|_{\infty}\mu(S)=\varepsilon\|y^*\|_{\infty}.
\end{align*}
Hence $I_{\phi^*}$ is Fr\'{e}chet differentiable at $x^*$.
\end{proof}

\begin{remark}Let $\phi:E\rightarrow\RX$ be proper, lower semicontinuous and convex.
By Fact~\ref{BLal:1} and $\phi^{**}=\phi$,
$\phi$ is differentiable everywhere on $E$ if and only if
$I_{\phi}$ is Fr\'{e}chet differentiable everywhere on $L^{\infty}_{E}(S, \mu)$.
\end{remark}

\subsection{Strong rotundity and stability}\label{ss:stab}

We may apply our results to an important optimization:

Let $(C_n)_{n\in\NN}$ and $C_{\infty}$ in $X$ be  closed convex sets, and let  $f:X\rightarrow\RX$ be a proper convex function with  weakly compact lower level sets.  We consider the following sequences of optimization problems (See \cite{BorLew1}.).
\begin{align*}
(P_n)&\qquad \qquad\qquad\qquad V(P_n):=\inf\big\{f(x)\mid x\in C_n\big\},\\
(P_{\infty})&\qquad\qquad \qquad\qquad V(P_{\infty}):=\inf\big\{f(x)\mid x\in C_{\infty}\big\}.
\end{align*}

\begin{fact}[Borwein and Lewis]\emph{(See \cite[Theorem~2.9(ii)]{BorLew1}.)}\label{BLal:4a}
Let $(C_n)_{n\in\NN}$ and $C_{\infty}$ in $X$ be closed convex sets,
and $f:X\rightarrow\RX$ be a proper convex function with  weakly compact lower level sets. Assume that \begin{align*}
\overline{\lim}^{\w}C_n\subseteq C_{\infty}\subseteq\bigcup_{m\geq1}\bigcap_{n\geq m}C_n.
\end{align*}
Then $V(P_n)\longrightarrow V(P_{\infty})$. If $V(P_{\infty})<+\infty$ and $f$ is strongly rotund, then
$(P_n)$ and $(P_{\infty})$ respectively have unique optimal solutions with $x_n$ and $x_{\infty}$, and $x_n\longrightarrow x_{\infty}$.
\end{fact}

In a typical application, $C_{n+1} \subseteq C_n$ may be nested polyhedral approximations to a convex set $C_\infty:= \bigcap C_n$, that is \emph{constructible} in the sense of \cite{BorVan}.  We look at the failure of strong rotundity in more detail in Section \ref{ApVisin}.

\section{Properties of Legendre functions and $I_{\phi}$}\label{s:voi}

\begin{proposition}\label{ProInteg:A1}
Let $\phi:E\rightarrow\RX$ be proper, lower semicontinuous and convex with $\inte\dom\phi\neq\varnothing$.
Then $\phi$ is essentially smooth if and only if
 $\partial I_{\phi}$ is single-valued.
\end{proposition}
\begin{proof}
``$\Rightarrow$":
 First, we show that $\partial I_{\phi}$ is single-valued.
Let $\{x^*,y^*\}\subseteq\partial I_{\phi}(x)$. Then by Fact~\ref{Intephi},
$x^*(s)\in \partial \phi(x(s))$ for almost all $s\in S$ and
$y^*(s)\in \partial \phi(x(s))$ for almost all $s\in S$.
Since $\partial \phi$ is single-valued. Then $x^*(s)=y^*(s)$ for almost all $s\in S$. Hence $x^*$ is equivalent to $y^*$ and then $x^*=y^*$.
Thus $\partial I_{\phi}(x)$ is single-valued.

``$\Leftarrow$": By Fact~\ref{legenFa:1}\ref{legenFa:1E1},
it suffices to show that $\partial \phi$ is single-valued.
Let $\{v_1^*, v^*_2\}\subseteq\partial \phi(v)$. Set $x(s):=v$,
 $x^*_1(s):=v^*_1$ and $x^*_2(s):=v^*_2$.
Then $x\in L^1_E (S,\mu)$, and $\{x_1^*, x^*_2\}\subseteq L_E^{\infty} (S,\mu)$.
Thus by Fact~\ref{Intephi},
$\{x_1^*, x^*_2\}\subseteq \partial I_{\phi}(x)$.
Since $\partial I_{\phi}$ is single-valued,
$x^*_1=x^*_2$ and hence $v^*_1=v^*_2$. Then $\partial \phi (x)$ is
single-valued and thus $\partial \phi$ is single-valued.
\end{proof}

\begin{corollary}\label{ProInteg:1}
Let $\phi:E\rightarrow\RX$ be proper, lower semicontinuous and convex with $\inte\dom\phi\neq\varnothing$.
Assume that
 $I_{\phi}$ is strictly convex on its domain
and that $\partial I_{\phi}$ is single-valued.
Then $\phi$ is Legendre.

\end{corollary}
\begin{proof}
By Fact~\ref{legenFa:2},
$\phi$ is strictly convex on its domain.
 Fact~\ref{legenFa:1}\ref{legenFa:1E3}
 implies that $\phi$ is essentially strictly convex.

Applying Proposition~\ref{ProInteg:A1},
$\phi$ is essentially smooth.
Combining the above results, $\phi$ is Legendre.
\end{proof}

\begin{corollary}\label{ProInteg:2}
Let $\phi:\RR\rightarrow\RX$ be proper, lower semicontinuous and convex  with $\inte\dom\phi\neq\varnothing$. Then
 $\phi$ is Legendre  if and only if
 $I_{\phi}$ is strictly convex on its domain and
 $\partial I_{\phi}$ is single-valued.

\end{corollary}
\begin{proof}
``$\Rightarrow$":
By Fact~\ref{legenFa:1}\ref{legenFa:1E3},
$\phi$ is strictly convex on $\inte\dom\phi$, and then
$\phi$ is strictly convex on $\dom\phi$. Hence $I_{\phi}$
 is strictly convex on its domain by Fact~\ref{legenFa:2}.
  By Proposition~\ref{ProInteg:A1}, $\partial I_{\phi}$ is single-valued.

``$\Leftarrow$": Applying Corollary~\ref{ProInteg:1} directly.
\end{proof}

Lemma~\ref{Ecbness} allows us to generalize \cite[Lemma~3.3]{BorLew1}.

\begin{lemma}\label{BLal:2}
Let $\phi:E\rightarrow\RX$ be proper, lower semicontinuous and convex,
 and let $x\in L^1_E(S,\mu)$. Assume that there exists a bounded closed subset $D$ of $\inte\dom \phi$
 such that $x(s)\in D$ almost everywhere on $S$.
 Then $\partial I_\phi(x)\neq \varnothing$.

\end{lemma}
\begin{proof} By the assumption, there exists a measurable subset $T$ of $S$ such that
$\mu(T)=\mu(S)$ and $x(s)\in D, \forall s\in T$.
By \cite[Lemma~2.6]{ph}, $\partial \phi$ is
upper semicontinuous on $D$. Thus, for every closed set
$C\subseteq E$, we have
$(\partial \phi)_C:= D \cap \big((\partial \phi)^{-1} C\big)$ is closed.
Thus
\begin{align*}
\big\{s\in T\mid \partial \phi(x(s))\cap C\neq\varnothing\big\}
=\big\{s\in T\mid  s\in x^{-1}\left[(\partial \phi)_C\right] \big\}
\end{align*}
is measurable. Hence $s\mapsto\partial \phi(x(s))$ is measurable on $T$.
Then by \cite[Theorem~14.2.1]{KleThom},
there exists a measurable selection $x^*(s)\in\partial \phi(x(s))$   everywhere on $T$.
Then $x^*(s)\in\partial \phi(x(s))$  almost everywhere on $S$ by $\mu (T)=\mu (S)$.
By  Lemma~\ref{Ecbness}, $\{x^*(s)\mid s\in T\}\subseteq \partial \phi(D)$ is bounded,  and
then $x^*(s)$ is bounded almost everywhere on $S$ since $\mu (T)=\mu (S)$.
Hence we have $x^*\in L_E^{\infty}( S,\mu)$.
\end{proof}

Let $m\in\NN$ and $x\in L^1_E(S,\mu)$, we define $S_m$
by
\begin{align}
S_m:=\big\{s\in S\mid x(s)\in m B_E\big\}.
\end{align}
Then we have $S_m\subseteq S_{m+1}$ and $S=\bigcup_{m\geq1} S_m$.
\begin{remark}\label{BLal:3}
Assume that $x\in L^1_E (S,\mu)$.
 Then $\mu(S_m^c)\downarrow 0$ and
$\mu(S_m)\uparrow \mu(S)$ when $m\longrightarrow \infty$.
\end{remark}

The proof of Proposition~\ref{BLal:4} was inspired by that of \cite[Lemma~3.6]{BorLew1}.
\begin{proposition}
\label{BLal:4} Let $\phi:E\rightarrow\RX$ be proper, lower semicontinuous and convex.
Suppose that $\phi^*$ is differentiable on $E$, and that $x_n\weakly x$ in $L^1_E(S,\mu)$ and $I_{\phi}(x_n)\longrightarrow I_{\phi}(x)<+\infty$. Assume that
 $x(s)\in\inte\dom \phi$ almost everywhere. Then $\|x_n-x\|_1\longrightarrow 0$.
\end{proposition}
\begin{proof}
Since $x_n\weakly x$, $D:=(x_n)_{n\in\NN}\cup\{x\}$ is weakly compact in
$L^1_E (S,\mu)$.
Let $\varepsilon>0$. By Fact~\ref{Dundcom:1},
 there exists $\delta>0$ such that
\begin{align}\label{BLal:4Eb1}
\int_C\big\|y(s)\big\| {\rm d}\mu(s)\leq\varepsilon,\quad\forall
\mu(C)\leq\delta,\,\forall y\in D.
\end{align}
Set $U:=\inte\dom \phi$. Then by Corollary~\ref{CoVitaliCo}, there exist a sequence $(z_n)_{n\in\NN}$ in $U$
and a sequence $(\delta_n)_{n\in\NN}$ in $\left[0,1\right]$
  such that
 \begin{align}\label{BLal:4E7}
U=\bigcup_{n\in\NN} B(z_n,\delta_n).
 \end{align}
 Set \begin{align}T_n:=\big\{s\in S\mid x(s)\in\bigcup^{n}_{k=1} B(z_k,\delta_k)\big\},
 \quad\forall n\in\NN. \label{BLal:4Ea1}
 \end{align}
Since $x(s)\in U$ almost everywhere on $S$, by\eqref{BLal:4E7},
 we have $\mu({T_n}^c)\downarrow 0$ and
$\mu(T_n)\uparrow \mu(S)$ when $n\longrightarrow \infty$.
Set $\widetilde{S_m}:=S_m\cap T_m$.
Then by Remark~\ref{BLal:3}, \begin{align}\label{BLal:4E5}
\mu(\widetilde{S_m}^c)=\mu(S_m^c\cup T_m^c)\downarrow 0\quad\text{and}\quad
\mu(\widetilde{S_m})\uparrow \mu(S)\quad\text{as}\quad m\longrightarrow \infty.
\end{align}
Then by \eqref{BLal:4Eb1}, there exists $N\in\NN$ such that
\begin{align}
\mu(\widetilde{S_m}^c)\leq\delta\quad\text{and}\quad
\int_{\widetilde{S_m}^c}\big\|x_n(s)-x(s) \big\|{\rm d}\mu(s)\leq2
\varepsilon,\quad\forall m\geq N, \forall n\in\NN.\label{BLal:4E1}
\end{align}
Then by Fact~\ref{legenFa:3},
\begin{align}\label{BLal:4E6}
x_n|_{\widetilde{S_m}}\weakly x|_{\widetilde{S_m}}\,\text{ in $L^1_E(\widetilde{S_m},\mu|_{\widetilde{S_m}}) $\quad and}\quad I^{\widetilde{S_m}}_{\phi}(x_n|_{\widetilde{S_m}})\longrightarrow
I^{\widetilde{S_m}}_{\phi} (x|_{\widetilde{S_m}})<+\infty.
\end{align}
By the definition of ${\widetilde{S_m}}$
and \eqref{BLal:4Ea1}, we have
\begin{align}
\big\{x(s)\mid s\in \widetilde{S_m}\big\}\subseteq mB_E\cap\big(\bigcup^{m}_{k=1} B(z_k,\delta_k)\big) .
\end{align}
By \eqref{BLal:4E7}, $mB_E\cap\big(\bigcup^{m}_{k=1} B(z_k,\delta_k)\big)$ is a  bounded closed subset of $\inte\dom\phi$.
Then by Lemma~\ref{BLal:2}, $\partial I^{\widetilde{S_m}}_{\phi} (x|_{\widetilde{S_m}})\neq\varnothing$.
Thus by Fact~\ref{BLal:1}, Fact~\ref{Intephi}, Fact~\ref{BorLewNorm:1} and \eqref{BLal:4E6}, we obtain that
\begin{align}
\int_{\widetilde{S_m}}\big\|x_n(s)-x(s)\big\|d \mu(s)\longrightarrow 0,\quad\text{as}\,\, n\longrightarrow0.\label{BLal:4E3}
\end{align}
Then we have
\begin{align}
\|x_n -x\|_1&=\int_{\widetilde{S_m}}\big\|(x_n(s))-(x(s))\big\|d \mu(s)+\int_{\widetilde{S_m}^c}\big\|(x_n(s))-(x(s))\big\|d \mu(s)\nonumber\\
&\leq\int_{\widetilde{S_m}}\big\|x_n(s)-x(s)\big\|d \mu(s)+
2\varepsilon,\quad\forall m\geq N\quad\text{(by \eqref{BLal:4E1}}.\label{BLal:4E4}
\end{align}
Taking $n\longrightarrow\infty$ in \eqref{BLal:4E4},
by \eqref{BLal:4E3}, $\limsup\|x_n -x\|_1\leq2\varepsilon$ and hence
$\|x_n-x\|_1\longrightarrow0$.
\end{proof}

We first prove a restrictive sufficient condition for strong rotundity.

\begin{theorem}
\label{BLal:5} Let $\phi:E\rightarrow\RX$ be proper, lower semicontinuous and convex with open domain.
Suppose that $\phi^*$ is differentiable on $E$. Then $I_{\phi}$ is strongly rotund on $L^1_E(S,\mu)$.
\end{theorem}
\begin{proof}
By Fact~\ref{phiEmth}, $\phi$ is essentially strictly convex. Since $\dom \phi$ is open, \cite[Proposition~3.3 and Proposition~1.11]{ph} implies that $\dom\partial\phi=\dom\phi$. Hence $\phi$ is strictly convex on $\dom\phi$.
Then by Fact~\ref{legenFa:2}, $I_{\phi}$ is strictly convex on its domain.
Since $\dom \phi^*=E$, by \cite[Corollary~2B]{Rock71PJM}, $I_{\phi}$ has weakly compact lower level sets.

Now we show $I_{\phi}$ has the Kadec property.
Let $x_n\weakly x\in\dom I_{\phi}$ in $L^1_E(S,\mu)$ and $I_{\phi}(x_n)\longrightarrow I_{\phi}(x)$. Since
$x\in \dom I_{\phi}$, $x(s)\in\dom \phi$ for almost all $s\in S$.
Since $\dom\phi=\inte\dom\phi$,
 $x(s)\in\inte\dom \phi$ almost everywhere. Then by Proposition~\ref{BLal:4},
  $\|x_n-x\|_1\longrightarrow 0$.

  Hence $I_{\phi}$ has the Kadec property and consequently $I_{\phi}$ is strongly rotund.
\end{proof}

 When the domain of $\phi$ is not open we have more work to do:

\begin{theorem}\label{Tsge2:1}
Let $\phi_i:\RR\rightarrow\RX$ be proper, lower semicontinuous and convex with $\inte\dom\phi_i\neq\varnothing$ for every $i=1,2,\cdots, d$.
Suppose that $\phi_i^*$ is differentiable on $\RR$  for every $i=1,2,\cdots, d$.
Let $\phi: E\rightarrow\RX$ be defined by $z:=(z_n)\in E\mapsto\sum_{i=1}^{d}\phi_i(z_i)$.
Then $I_{\phi}$ is strongly rotund on $L^1_E(S,\mu)$.
\end{theorem}
\begin{proof}
We have $\phi$ is proper lower semicontinuous and convex. Let $i\in\{1,\cdots,d\}$. By Fact~\ref{phiEmth}, $\phi_i$ is essentially strictly convex. Then
$\phi_i$ is strictly convex on $\inte\dom\phi_i$.
Hence $\phi_i$ is strictly convex on its domain, so is $\phi$.
Then by Fact~\ref{legenFa:2}, $I_{\phi}$ is strictly convex on its domain.
By the assumption, $\phi^*=\sum_{i=1}^{d}\phi_i^*$, hence $\phi^*$ is differentiable everywhere on $E$. Then by \cite[Corollary~2B]{Rock71PJM}, $I_{\phi}$ has weakly compact lower level sets.

Now we show $I_{\phi}$ has the Kadec property.
Let $x_n\weakly x\in\dom I_{\phi}$ in $L^1_E(S,\mu)$ and $I_{\phi}(x_n)\longrightarrow I_{\phi}(x)$. Since $x\in\dom I_{\phi}$, $x(s)\in \dom\phi$ for all most $s\in S$. We can and do suppose that $x(s)\in\dom \phi$ for all $s\in S$.

We let $x(s):= \big(x_1 (s),\cdots, x_d(s)\big)$ and $x_n(s):= \big(x_{n,1} (s),\cdots, x_{n,d}(s)\big)$.
Now we claim that
\begin{align}
\int_S\big|x_{n,i}(s)-x_i(s)\big|{\rm
d}\mu(s)\longrightarrow0,\quad\forall
i\in\{1,\cdots,d\}.\label{Tsge2:e1}
\end{align}
Fix $i\in\{1,\cdots,d\}$. Since $\inte\dom\phi_i\neq\varnothing$, there exist $\alpha\in\RR\cup \{-\infty\}$ and $\beta\in\RR\cup\{+\infty\}$ such that $\alpha<\beta$ and $\inte\dom\phi_i=\left]\alpha,\beta\right[$.
 We set
\begin{align*}
S_{\alpha}&:=\big\{s\in S\mid x_{i}(s)=\alpha\big\}\\
S_{\beta}&:=\big\{s\in S\mid x_{i}(s)=\beta\big\}\\
S_{\inte}&:=\big\{s\in S\mid x_{i}(s)\in\inte\dom \phi_i\big\}.
\end{align*}
Then $S_{\alpha}, S_{\beta}$ and $S_{\inte}$ are measurable sets.
Given $y(s)=(y_i(s))_{i=1}^{d}\in L^1_E(S,\mu)$. Set $\widetilde{y}$ by
\begin{align*}
\widetilde{y}(s)&:=\big(y_{1}(s),\cdots,y_{i -1}(s), y_{i+1}(s),\cdots, y_d(s)\big).
\end{align*}
Now we show that \begin{align}
\widetilde{x_n}\weakly \widetilde{x}\quad\text{in}\, L^1_{R^{d-1}}(S,\mu).\label{BoShE:ta4}
\end{align}

Let $v^*(s)\in L^{\infty}_{R^{d-1}}(S,\mu)$.
For convenience, we write \begin{align*}v^*(s) =\big(v_{1}(s),\cdots,v_{i -1}(s), v_{i+1}(s),\cdots, v_d(s)\big).
\end{align*}
Then we define ${w^*}$ by \begin{align*}w^*(s):= \big(v_{1}(s),\cdots,v_{i -1}(s),0, v_{i+1}(s), \cdots,v_d(s)\big).
 \end{align*}
 Then $w^*\in L^{\infty}_E(S,\mu)$ and
$\langle \widetilde{x_n}, v^*\rangle=\langle x_n, w^*\rangle\longrightarrow
\langle x, w^*\rangle=\langle \widetilde{x}, v^*\rangle$.
Hence $\widetilde{x_n}\weakly  \widetilde{x}$ and thus \eqref{BoShE:ta4} holds.

Similarly, we have
 \begin{align}
x_{n,i}\weakly x_{i}\quad\text{in}\, L^1_{R}(S,\mu).\label{BoShE:ta5}
\end{align}

Then by \eqref{BoShE:ta4}, \eqref{BoShE:ta5} and Fact~\ref{legenFa:3},
\begin{align}
\widetilde{x_n}|_{S_\gamma}\weakly  \widetilde{x}|_{S_\gamma}\quad\text{and}\quad
x_{n,i}|_{S_\gamma}\weakly  x_{i}|_{S_\gamma},\quad \gamma\in\{\alpha,\beta,\inte\}.\label{BoShE:ta6}
\end{align}
Since $I_{\phi}(x_n)\longrightarrow I_{\phi}(x)<+\infty$,  we have $I_{\phi}(x_n)<+\infty$ and hence $x_n(s)\in\dom \phi$ for all
almost $s\in S$ when $n$ is larger enough. Thus, we can and do assume that $x_{n,i}(s)\in\dom\phi_i$ for all $n\in\NN, s\in S$.
Since $S=S_{\alpha}\cup S_{\beta}\cup S_{\inte}$,
we have
\begin{align}
&\int_{S}\big|x_{n,i}(s)-x_i(s)\big|{\rm d}\mu(s)\nonumber\\
&=\int_{S_{\alpha}}\big|x_{n,i}(s)-x_i(s)\big|{\rm
d}\mu(s)+\int_{S_{\beta}}\big|x_{n,i}(s)-x_i(s)\big|{\rm d}\mu(s)
+\int_{S_{\inte}}\big|x_{n,i}(s)-x_i(s)\big|{\rm d}\mu(s)\nonumber\\
&=\int_{S_{\alpha}}\big(x_{n,i}(s)-x_i(s)\big){\rm
d}\mu(s)+\int_{S_{\beta}}\big(x_i(s)-x_{n,i}(s)\big){\rm d}\mu(s)
+\int_{S_{\inte}}\big|x_{n,i}(s)-x_i(s)\big|{\rm
d}\mu(s).\label{BoShE:ta07}
\end{align}
By \eqref{BoShE:ta6},
\begin{align}
\int_{S_{\alpha}}\big(x_{n,i}(s)-x_i(s)\big){\rm
d}\mu(s)+\int_{S_{\beta}}\big(x_i(s)-x_{n,i}(s)\big){\rm
d}\mu(s)\longrightarrow0. \label{BoShE:ta7}
\end{align}
Now we show that
\begin{align}
\int_{S_{\inte}}\big|x_{n,i}(s)-x_i(s)\big|{\rm d}\mu(s)
\longrightarrow0.\label{BoShE:ta8}
\end{align}
If $\mu(S_{\inte})=0$, clearly, \eqref{BoShE:ta8} holds. Now we assume that
$\mu(S_{\inte})\neq0$.
We define $\psi: \RR^{d-1}\rightarrow\RX$ by
$z:=(z_1,z_2,\cdots,z_{i-1},z_{i+1},\cdots,z_d)\mapsto\sum_{j\neq i}\phi_j(z_j)$.
Then by Fact~\ref{Intephi}, $I^{S_{\inte}}_{\psi}$
and $I^{S_{\inte}}_{\phi_i}$ are proper lower semicontinuous and convex.
Then by Remark~\ref{reInteRoc} and \eqref{BoShE:ta6},
\begin{align}\liminf \int_{S_{\inte}}\psi\big(\widetilde{x_n}|_{S_{\inte}}(s)\big){\rm d}\mu(s)=
\liminf I^{S_{\inte}}_{\psi}(\widetilde{x_n}|_{S_{\inte}})
&\geq I^{S_{\inte}}_{\psi}(\widetilde{x}|_{S_{\inte}}),\nonumber\\
\liminf I^{S_{\inte}}_{\phi_i}(x_{n,i}|_{S_{\inte}})
&\geq I^{S_{\inte}}_{\phi_i}(x_{i}|_{S_{\inte}}).\label{BoShE:ta09}
\end{align}
We first show that
\begin{align}
\int_{S_{\inte}}\phi_i\big(x_{n,i}|_{S_{\inte}}(s)\big){\rm
d}\mu(s)\longrightarrow
\int_{S_{\inte}}\phi_i\big(x_i|_{S_{\inte}}(s)\big){\rm
d}\mu(s)<+\infty. \label{BoShE:ta9}
\end{align}
\allowdisplaybreaks
By  Fact~\ref{legenFa:3}, we have
\begin{align}
I^{S_{\inte}}_{\phi}(x_n|_{S_{\inte}})\longrightarrow I^{S_{\inte}}_{\phi}(x|_{S_{\inte}})<+\infty.
\label{Tsge2:e3}
\end{align}
Then we have
\begin{align*}
&\limsup I^{S_{\inte}}_{\phi_i}(x_{n,i}|_{S_{\inte}})=\limsup
\int_{S_{\inte}}\phi_i\big(x_{n,i}|_{S_{\inte}}(s)\big){\rm d}\mu(s)
=\limsup \big(I^{S_{\inte}}_{\phi}(x_n|_{S_{\inte}})-I^{S_{\inte}}_{\psi}(\widetilde{x_n}|_{S_{\inte}})\big)\\
&=\lim I^{S_{\inte}}_{\phi}(x_n|_{S_{\inte}})-\liminf I^{S_{\inte}}_{\psi}(\widetilde{x_n}|_{S_{\inte}}))\\
&\leq I^{S_{\inte}}_{\phi}(x|_{S_{\inte}})- I^{S_{\inte}}_{\psi}(\widetilde{x}|_{S_{\inte}})\quad\text{(by \eqref{Tsge2:e3} and \eqref{BoShE:ta09})}\\
&= I_{\phi_i}^{S_{\inte}}(x_i|_{S_{\inte}})<+\infty\quad\text{(since $I^{S_{\inte}}_{\phi}(x|_{S_{\inte}})<+\infty$ and $I_{\psi}^{S_{\inte}}(\widetilde{x}|_{S_{\inte}})>-\infty$ by \eqref{Tsge2:e3} and Fact~\ref{Intephi})}.
\end{align*}
Then combining with \eqref{BoShE:ta09}, we have
$\limsup I^{S_{\inte}}_{\phi_i}(x_{n,i}|_{S_{\inte}})\leq I^{S_{\inte}}_{\phi_i}(x_{i}|_{S_{\inte}})\leq
\liminf I^{S_{\inte}}_{\phi_i}(x_{n,i}|_{S_{\inte}})$ and thus \eqref{BoShE:ta9} holds.

By \eqref{BoShE:ta6}, \eqref{BoShE:ta9} and
Proposition~\ref{BLal:4}, we have
$\int_{S_{\inte}}\big|x_{n,i}(s)-x_i(s)\big|{\rm d}\mu(s)
\longrightarrow0$ and hence \eqref{BoShE:ta8} holds.

Combining  \eqref{BoShE:ta8}, \eqref{BoShE:ta7} and
\eqref{BoShE:ta07}, we have $\int_{S}|x_{n,i}(s)-x_i(s)|{\rm
d}\mu(s)\longrightarrow 0$ and hence \eqref{Tsge2:e1} holds.

Then by \eqref{Tsge2:e1},
\begin{align*}
&\|x_n- x\|_1\leq \int_S\sum_{i=1}^{d} \big|x_{n,i}(s)-
x_i(s)\big|{\rm d}\mu(s)\longrightarrow0.
\end{align*}
Hence $x_n\longrightarrow x$ and hence $I_\phi$ has the Kadec property.

Combining the above
results, $I_\phi$ is strongly rotund  in $L^1_E(S,\mu)$.
\end{proof}

\begin{remark}It is noted in \cite{BorLew1} that strongly rotund functions with points of continuity can only exist on reflexive spaces. Moreover, strongly rotund integral functions on $L^1_E(S,\mu)$  are a useful surrogate for strongly rotund renorms which always exist in the reflexive setting.\qede\end{remark}

\section{Examples and applications}\label{s:ExA}

Below we use the convention that $0\log0=0$.

\begin{example}\label{ExmP:E1}
By applying Theorem~\ref{BLal:5} and Theorem~\ref{Tsge2:1}, we can obtain many functions
$\phi$ such that $I_{\phi}$ is strongly rotund. Seven examples  follow
\begin{enumerate}
\item\label{BoShE:1}
Let $f:\RR\rightarrow\RX$ be defined by
\begin{align*}
f(x)=\begin{cases}x\log x-x,\,&\text{if}\, x\geq0;\\
+\infty,\,&\text{otherwise}
\end{cases}\quad \forall x\in\RR.
\end{align*}
Let $\phi:E\rightarrow\RX$ be defined by
\begin{align*}\phi(x):=\sum_{i=1}^{d}f(x_i),\quad \forall x=(x_n)\in E.
\end{align*}
Then  $I_\phi$ is the \emph{Boltzmann-Shannon entropy}.

\item \label{BoShE:2}
Let $f:\RR\rightarrow\RX$ be defined by
\begin{align*}
f(x)=\begin{cases}x\log x+(1-x)\log(1-x),\,&\text{if}\, 0\leq x\leq1;\\
+\infty,\,&\text{otherwise}
\end{cases}\quad \forall x\in\RR.
\end{align*}
Let $\phi:E\rightarrow\RX$ be defined by
\begin{align*}\phi(x):=\sum_{i=1}^{d}f(x_i),\quad \forall x=(x_n)\in E.
\end{align*}
Then $I_\phi$ is the \emph{Fermi-Dirac entropy}

\item \label{Sephfun:1}$\phi (x)=\tfrac{1}{p}\|x\|^p,\quad \forall x\in E,\quad\text{where $p>1$}$.
\item \label{Sephfun:2}$\phi(x)=\begin{cases}\sum_{i=1}^{d}-\log(\cos x_i),\,&\text{if}\,
x\in\left]-\tfrac{\pi}{2}, \tfrac{\pi}{2}\right[\times\cdots\times\left]-\tfrac{\pi}{2}, \tfrac{\pi}{2}\right[\\
+\infty,&\text{otherwise}
\end{cases}
\quad\forall x=(x_n)\in E$.
\item \label{Sephfun:3}$\phi (x)=\sum_{i=1}^{d}\cosh x_i,\quad \forall x=(x_n)\in E $.
\item \label{Sephfun:4}$\phi (x)=\begin{cases}
\sum_{i=1}^{d}\big(x_i\tanh^{-1}x_i+\tfrac{1}{2}\log(1-x^2_i)\big),
\,&\text{if}\,
|x_i|<1,\, \forall i;\\
+\infty,&\text{otherwise}
\end{cases}
\quad \forall x=(x_n)\in E $.

\item \label{Sephfun:5}$\phi (x)=\begin{cases}-\tfrac{1}{1-\|x\|^2},\quad&\text{if}\,\|x\|<1;\\
+\infty,&\text{otherwise}
\end{cases}
\quad \forall x\in E $.

\end{enumerate}
\end{example}

\begin{proof}
\ref{BoShE:1}:
Clearly,
$f$ is proper lower semicontinuous and  convex.
By \cite[Table~2.1, pp.~45]{BorVan}, $f^*(x)=\exp(x),\forall x\in\RR$.
Then directly apply Theorem~\ref{Tsge2:1}.

\ref{BoShE:1}:
Clearly,
$f$ is proper lower semicontinuous and  convex.
By \cite[Table~2.1, pp.~45]{BorVan}, $f^*(x)=\log\big(1+\exp(x)\big),\forall x\in\RR$.
Then directly apply Theorem~\ref{Tsge2:1}.

\ref{Sephfun:1}:
Clearly,
$\phi$ is continuous and  convex with full domain.
We have $\phi^*=\tfrac{1}{q}\|\cdot\|^q$, where $\tfrac{1}{q}+\tfrac{1}{p}=1$ and
$(\phi^*)'=(\|\cdot\|)^{q-2}\small\circ\Id$. Hence $\phi^*$ is differentiable everywhere on $E$.
Then directly apply Theorem~\ref{BLal:5}.

\ref{Sephfun:2}:
Let $f(x):=\begin{cases}-\log(\cos x),\,&\text{if}\,
x\in\left]-\tfrac{\pi}{2}, \tfrac{\pi}{2}\right[\\
+\infty,&\text{otherwise}
\end{cases}$.
By \cite[Table~2.1, pp.~45]{BorVan}, we have $f$ is proper lower semicontinuous and convex,
and
$f^*(x)=x\tan^{-1}x-\tfrac{1}{2}\log(1+x^2)\big), \forall x\in\RR$.
Hence $f^*$ is differentiable everywhere on $\RR$.
Then directly apply Theorem~\ref{Tsge2:1}.

\ref{Sephfun:3}:
Let $f(x):=\cosh (x)$.
By \cite[Table~2.1, pp.~45]{BorVan}, we have $f$ is continuous and convex, and
 $f^*(x)=x\sinh^{-1}x-\sqrt{1+x^2}, \forall x\in X$.
Hence $f^*$ is differentiable everywhere on $\RR$.
Then directly apply Theorem~\ref{Tsge2:1}.

\ref{Sephfun:4}:
Let $f(x):=\begin{cases}
x\tanh^{-1}x+\tfrac{1}{2}\log(1-x^2)\big),
\,&\text{if}\,
|x|<1;\\
+\infty,&\text{otherwise}
\end{cases}$.
By \cite[Table~2.1, pp.~45]{BorVan}, we have $f$ is proper lower semicontinuous and convex,
and
$f^*(x)=\log(\cosh x)$.
Thus $f^*$ is differentiable everywhere on $\RR$.
Then directly apply Theorem~\ref{Tsge2:1}.

\ref{Sephfun:5}: Clearly, $\dom\phi$ is open. By \cite[Example~6.4]{BBC1}, $\phi$ is proper lower semicontinuous and convex, and
$\phi^*$ is differentiable everywhere on $E$.
Then directly apply Theorem~\ref{BLal:5}.
\end{proof}

\begin{example}
Let $(a^*_n)_{n\in\NN}$ be a sequence in $L^{\infty}_E(S,\mu)$ and let $(b_n)_{n\in\NN}$ be a sequence  in $\RR$. Let $\phi:X\rightarrow\RX$ be one of the functions given in Example~\ref{ExmP:E1}.

 We consider the following optimization problems (See \cite[page~196]{BorLew2}.).
\begin{align*}
(P_n)&\qquad \qquad\qquad\qquad \begin{cases}V(P_n):=&\inf I_{\phi}(x)\\
\text{subject to}\quad& \langle a^*_i, x\rangle= b_i,\quad i=1,\cdots,n\\
&x\in L^1_E(S,\mu)
\end{cases}\\
(P_{\infty})&\qquad\qquad \qquad\qquad
\begin{cases}V(P_{\infty}):=&\inf I_{\phi}(x)\\
\text{subject to}\quad& \langle a^*_i, x\rangle= b_i,\quad i=1,\cdots,n,n+1,\cdots\\
&x\in L^1_E(S,\mu).
\end{cases}
\end{align*}Then we have $V(P_n)\longrightarrow V(P_{\infty})$.
If, moreover,  $V(P_{\infty})<+\infty$, then
$(P_n)$ and $(P_{\infty})$ respectively have unique optimal solutions with $x_n$ and $x_{\infty}$, and $x_n\longrightarrow x_{\infty}$. \qede
\end{example}

\begin{proof}
Set \begin{align*}C_n:&=\big\{x\in L^1_E(S,\mu)\mid \langle a^*_i, x\rangle= b_i,\quad i=1,\cdots,n\big\}\\
C_{\infty}:&=\big\{x\in L^1_E(S,\mu)\mid \langle a^*_i, x\rangle= b_i,\quad i=1,\cdots,n,n+1,\cdots\big\}.
\end{align*}
Then we have $C_1\supseteq C_2\supseteq\ldots\supseteq C_n \supseteq\ldots$. Thus,
$\overline{\lim}^{\w}C_n\subseteq C_{\infty}$ and $C_{\infty}=\bigcap_{n\geq1} C_n\subseteq
\bigcup_{m\geq1}\bigcap_{n\geq m} C_n$. We finish with a direct application of Example~\ref{ExmP:E1} and Fact~\ref{BLal:4a}.
\end{proof}

We  next revisit  a function $\phi$ given in \cite{BorLew3}
such that $I_\phi$ is not strongly rotund but $\phi$ is everywhere strictly convex.
\begin{example}[Borwein and Lewis]\label{lorm:1}
Let $\phi(x):=\begin{cases}
-\log x,
\,&\text{if}\,
x>0;\\
+\infty,&\text{otherwise}
\end{cases}, \forall x\in\RR$. Let $S=\left[0,1\right]$  and $\mu$ be Lebesgue measure.

Then $I_\phi$ is the \emph{Burg entropy}, and
$\phi^*(x)=\begin{cases}
-1-\log(-x),
\,&\text{if}\,
x<0;\\
+\infty,&\text{otherwise}
\end{cases}, \forall x\in\RR$.
 However,  $I_\phi$ does not have weakly compact lower
level sets (See \cite[page~258]{BorLew3}.). Hence $I_\phi$ is not strongly rotund. \qede
\end{example}

\section{Watson integral and Burg entropy nonattainment}\label{sec:burg}
Let $S=\left[0,1\right]\times \left[0,1\right]\times\left[0,1\right]$  and $\mu$ be Lebesgue measure,
 and let $\phi$ be defined as in Example~\ref{lorm:1}.
Consider the
perturbed \emph{Burg entropy}\, minimization problem
\begin{align*} \begin{cases} &\inf
I_{\phi}(x)\\
\text{subject to}\quad&
\int_S x(s) {\rm d}\mu(s)=1\\
&\int_S x(s)\cos\left(2\pi s_k\right){\rm d}\mu(s)=\alpha,\quad k=1,2,3\\
&x\in L^1_{\RR}(S, \mu),
\end{cases}
\end{align*}
where $s:=(s_1, s_2, s_3)$ and $d \mu (s):=d s_1 d s_2 d s_3$.
Then the above  problem is equivalent to the following.
\begin{align*} \hspace*{-1.52pc} &v\left(\alpha \right) :=
\sup_{0\leq p\in L^1_{\RR}(S,\mu)}  \bigg\{\int_S \log \left( p(x_1,x_2,x_3)\right)  \bigg| \int_S p(x_1,x_2,x_3) dx_1dx_2
dx_3=  1,\\ \,\hfill &\mbox{and for}\ \,
k=1,2,3, \, \int_S p(x_1,x_2,x_3)\cos
\left(2\pi x_k\right) d x_1dx_2 dx_3 =\alpha \bigg\}, 
\end{align*}
maximizing the log of a density $p$ with given mean, and with the
first three cosine moments fixed at  a parameter value  $0 \le
\alpha <1$.
 It transpires that there is a parameter value
$\overline \alpha$ such that below and at that value $v(\alpha)$
is attained, while above it is finite but unattained. This is
interesting, because:

 The general method---maximizing $\int_
S\log\left(p(s)\right)\,{\rm d}\mu(s)$  subject to a finite number
of trigonometric moments---is frequently used.  In one or two
dimensions, such spectral problems are always attained when
feasible.

There is no easy way to see that this problem qualitatively
changes at $\overline \alpha$, (by \cite[Eqs.~(5.8)\&(5.10)]{borweinlewis1993})  but we can get an idea by
considering
$$\overline{p}\left(x_1,x_2,x_3\right)=\frac{1/W_1}{3-\sum_1^3\cos\left(2\pi
x_i\right)},$$ and checking that this is feasible for
$$\overline \alpha  = 1-1/(3W_1) \approx 0.340537329550999142833$$
in terms of the first Watson integral, $W_1:=\int_{-\pi}^{\pi}\int_{-\pi}^{\pi}\int_{-\pi}^{\pi}\!\! \tfrac{1}{3-\cos(x_1)-\cos(x_2)-\cos(x_3)} d x_1dx_2 dx_3 $ (See \cite[Item~20, page~117]{BBG}
  and \cite{JZ} for more information about
$W_1$.). By using
Fenchel duality \cite{BorVan} one can show that this $\overline p$ is
optimal.

 Indeed, for all $\alpha \ge 0$ the only possible optimal
solution is of the form
$$\overline{p}_\alpha\left(x_1,x_2,x_3\right)
=\frac{1}{\lambda_\alpha^0-\sum_1^3 \lambda_\alpha^i
\cos\left(2\pi x_i\right)},$$ for  some real numbers
$\lambda_\alpha^i$.  Note that we have four coefficients to
determine; using the four constraints we can solve for them.
Let $W_1(w)$ be the generalized Watson integral, i.e.,
$W_1(w):=\int_{-\pi}^{\pi}\int_{-\pi}^{\pi}\int_{-\pi}^{\pi}\!\! \tfrac{1}{3-w\big(\cos(x_1)+\cos(x_2)+\cos(x_3)\big)} d x_1dx_2 dx_3$ (See \cite[Item~21(e), page~120]{BBG} and \cite{JZ} for more information about
$W_1(w)$.).

For
$0 \le \alpha \le \overline{\alpha}$, the precise form is
parameterized by  the generalized Watson integral:
$$\overline{p}_\alpha\left(x_1,x_2,x_3\right)
=\frac{1/W_1(w)}{3-\sum_1^3 w\cos\left(2\pi x_i\right)},$$ and
$\alpha=1-1/(3 W_1(w))$, as $w$ ranges from zero to one.

Note
also that $W_1(w)=\pi^3\,\int_0^\infty\, I_0^3(w\,t)\,e^{-3t}\,dt$
allows one to quickly obtain $w$ from $\alpha$ numerically.  For
$\alpha> \overline{\alpha}$, no feasible reciprocal polynomial can
stay positive.  Full details are given in \cite[Example~4, pp.~264-265]{borweinlewis1993}.

\section{Applications of Visintin's Theorem}\label{ApVisin}

Visintin's Theorem \cite[Theorem~3(i)]{Vin} on norm convergence of sequences converging weakly to an extreme point,
allows for a very efficient proof of the Kadec property for integral functionals.
Indeed, using Fact~\ref{ExVis:1}, we arrive at the following.
\begin{fact}[Visintin]\label{Visin:1}
\emph{(See 
\cite[Theorem~3(i)]{Vin}.)}
Let $\phi:E\rightarrow\RX$ be proper, lower semicontinuous and strictly convex.
Then
$I_{\phi}$ has the Kadec property.
\end{fact}

\begin{remark}
In the proofs of
Theorem~\ref{BLal:5} and Theorem~\ref{Tsge2:1}, we can also apply Visintin Theorem (see Fact~\ref{Visin:1}) to show that $I_\phi$ has the Kadec property.
\end{remark}
\begin{example}
Let $\phi$ be defined as in Example~\ref{lorm:1}. Then $I_{\phi}$ has the Kadec property.
Indeed, since $\phi$ is proper, lower semicontinuous and strictly convex,
it follows from Fact~\ref{Visin:1}  that $I_\phi$ has the Kadec property. \qede
\end{example}

\begin{theorem}[Strong rotundity]\label{Vistprp:2}
Let $\phi:E\rightarrow\RX$ be proper, lower semicontinuous and convex.
Suppose that $\phi$ is strictly convex on its domain and $\phi^*$ is differentiable on $E$. Then $I_{\phi}$ is strongly rotund on $L^1_E(S,\mu)$.
\end{theorem}

\begin{proof}
By Fact~\ref{legenFa:2}, $I_{\phi}$ is strictly convex on its domain.
Since $\dom \phi^*=E$, by \cite[Corollary~2B]{Rock71PJM}, $I_{\phi}$ has weakly compact lower level sets.
Visintin Theorem (see Fact~\ref{Visin:1}) implies that $I_\phi$ has the Kadec property.
Hence $I_\phi$ is strongly rotund.
\end{proof}

\begin{remark}
We cannot remove the assumption of  strict convexity of  $\phi$ in Theorem~\ref{Vistprp:2}.
For example, let $\phi:\RR^2\rightarrow\RX$ be defined by
\begin{align*}
(x,y)\mapsto\begin{cases}
-(xy)^{\tfrac{1}{4}},
\,&\text{if}\,\,
0\leq x\leq1,\,0\leq y\leq1;\\
+\infty,&\text{otherwise}
\end{cases}.
\end{align*} Then $\phi$ is proper lower semicontinuous and convex.
By \cite[Exercise~5.3.10, page~249]{BorVan}, $\phi$ is not strictly convex on its domain although $\phi^*$ is differentiable everywhere on $\RR^2$. Hence $I_{\phi}$ is not strongly rotund.
\end{remark}

\begin{remark}Let $\phi:E\rightarrow\RX$ be proper, lower semicontinuous and convex.
Suppose that  $\phi^*$ is differentiable on $E$.
Assume that $E$ is one-dimensional or $\dom\phi=\dom \partial \phi$ (for example,
 $\dom\phi$ is open), by Fact~\ref{phiEmth} and Fact~\ref{legenFa:1}\ref{legenFa:1E3},
  the differentiability of $\phi^*$ implies that the strictly convexity of $\phi$. Thus we can remove the assumption of the strictly convexity of  $\phi$ in Theorem~~\ref{Vistprp:2} under this constraint.
\end{remark}

\begin{example}Let $F$
be the Euclidean space that consists of  all symmetric $d\times d$ matrices with the inner product $\langle M, N\rangle=\tr(MN)$ (for every $M, N\in F$),
where $\tr(M)$ is the trace of the matrix $M$. Let $F_{++}$ be the set of symmetric $d\times d$ positive definite  matrices.
We define $\phi$ on $F$ by $M\mapsto\phi(M):=\begin{cases}
-\log \det(M),
\,&\text{if}\,
M\in F_{++};\\
+\infty,&\text{otherwise}
\end{cases}$, where $\det(M)$ is the determinant of the matrix $M$.
Then $I_{\phi}$ has the Kadec property in $L^1_{F}(S,\mu)$.
\end{example}

\begin{proof}
By \cite[Proposition~3.2.3, page~100]{BorVan},
$\phi$ is proper lower semicontinuous and strictly convex.  Then by Fact~\ref{Visin:1},
$I_{\phi}$ has the Kadec property in $L^1_{F}(S,\mu)$.
\end{proof}

\section{Convergence in measure}\label{sec:meas}

Recall that $S$ is an arbitrary non-trivial set and that $(S,\mu)$ is a complete finite measure space (with
nonzero measure $\mu$).
Let $(x_n)_{n\in\NN}$ and $x$ be in $L^1_E(S,\mu)$. We say $(x_n)_{n\in\NN}$ \emph{converges to $x$ in measure} if for every $\eta>0$,
$\lim\mu\big\{s\in S\mid \|x_n(s)-x(s)\|\geq\eta\big\}=0$.
We say $(x_n)_{n\in\NN}$ \emph{converges to $x$ $\mu$-- uniformly} if for every $\varepsilon>0$, there exists a measurable
subset $T$ of $S$ such that $\mu(T)<\varepsilon$ and $(x_n)_{n\in\NN}$ converges uniformly to $x$ on $T^c$.

Let $(x_n)_{n\in\NN}$ and $x$ be in $L^1_E(S,\mu)$. Then
$(x_n)_{n\in\NN}$ strongly converges to $x$ if and only if
$(x_n)_{n\in\NN}$ converges to $x$ in measure and $(x_n)_{n\in\NN}$
also weakly converges to $x$ (see \cite[Lemma~1 and Lemma~2]{Vin}).
Thus, for aa strictly convex integrand, Theorem \ref{Vistprp:2}
shows that weak convergence must fail whenever measure convergence
holds and strong convergence does not follow.

The following is another sufficient condition  for a  sequence
convergent in measure to be strongly convergent.
\begin{fact}\emph{See (\cite[Theorem~3.6, page~122]{DunSch})}\label{vicon:1}~ Let $(x_n)_{n\in\NN}$ be in $L^1_E(S,\mu)$ and $x:S\rightarrow E$.
Then $x\in L^1_E(S,\mu)$ and $x_n\longrightarrow x$ if and only if the following conditions hold:
\begin{enumerate}
 \item $(x_n)_{n\in\NN}$ converges to $x$ in measure.
\item $\lim_{\mu(E)\rightarrow0}\int_E\|x_n(s)\|{\rm d}\mu(s)=0$\quad\text{uniformly in $n$}.
\end{enumerate}
\end{fact}
See \cite{Bal} for more information on the relationships between weak, measure and strong convergence.

\begin{fact}\emph{(See \cite[Corollary~3.3, page~145]{DunSch}.)}\label{vicon:2} Let $(x_n)_{n\in\NN}$ and $x$ be in $L^1_E(S,\mu)$.
Assume that $(x_n)_{n\in\NN}$ converges to $x$ in measure. Then there exists a subsequence of $(x_n)_{n\in\NN}$ that  converges to $x$ $\mu$--uniformly.
\end{fact}

Let $f:X\rightarrow\RX$ be lower semicontinuous at  $x_0\in \dom f$. Then the \emph{Clarke-Rockafellar directional derivative}
of $f$ at $x_0$ is defined
\begin{align*}
f^{\uparrow}(x_0;v):=\sup_{\varepsilon>0}\limsup_{t\downarrow 0,\,
x\rightarrow_{f} x_0}\inf_{\|u-v\|\leq\varepsilon}\tfrac{f(x+tu)-f(x)}{t},\quad
\forall v\in X,
\end{align*} where $x\rightarrow_{f} x_0$ means that $x\longrightarrow x_0$ and
$f(x)\longrightarrow f(x_0)$.
Then the \emph{Clarke subdifferential} of $f$ at $x_0$ is defined by
\begin{align*}
\partial_C f(x_0):=\big\{x^*\in X^*\mid \langle x^*, v\rangle\leq f^{\uparrow}(x_0;v),\, \forall v\in X\big\}.
\end{align*}
If $f$ is also convex, then $\partial f=\partial_C f$ (see \cite[Theorem~3.2.4(ii)]{Zalinescu}).

We shall need the following mean value theorem:

\begin{fact}[Zagrodny]\label{mvat:1}\emph{(See \cite[Theorem~3.2.5]{Zalinescu}.)}
Let $f:X\rightarrow\RX$ be  proper  lower semicontinuous.
Let $x,y\in\dom f$.  Then there exist a sequence $(x_n)_{n\in\NN}$ and $z\in\left[x,y\right]$ and $x^*_n \in\partial_C f(x_n)$
such that $x_n\longrightarrow z$, $f(x_n)\longrightarrow f(x)$ and
\begin{align*}
f(y)-f(x)\leq\liminf\langle y-x, x^*_n\rangle.
\end{align*}
\end{fact}

We are now ready for two results showing  when convergence in measure of a sequence $(x_n)_{n\in\NN}$ allows us to deduce convergence of $\left(I_\phi(x_n)\right)_{n\in\NN}.$ This is useful if one thinks of $I_\phi$ as a measurement of a reconstruction $x_n$  for a member of a sequence which may not be norm convergent to the underlying signal $x$.

\begin{theorem}[\textbf{Preservation of convergence in measure, I}]\label{Prpvitc:1} Let $\phi:E\rightarrow\RX$ be continuous. Let $(x_n)_{n\in\NN}$ and $x$ be in $L^1_E(S,\mu)$ such that $(x_n)_{n\in\NN}$ converges to $x$ in measure. Assume that there exists $M>0$ such that $|\phi(v)|\leq M$ for all $v\in E$.
Suppose that one of the following conditions holds.\begin{enumerate}
\item $x\in L^{\infty}_E(S,\mu)$; or if
\item $\phi$ is uniformly continuous, in particular, when $\phi$ is globally Lipschitz.
\end{enumerate}
Then $\int_S\big|\phi\big(x_n(s)\big)-\phi\big(x(s)\big)\big|{\rm
d}\mu(s)\longrightarrow 0$.
 Consequently, $I_\phi (x_n)\longrightarrow I_\phi (x)$.

\end{theorem}

\begin{proof} We first  assume that $x\in L^{\infty}_E(S,\mu)$.
Suppose to the contrary that
$\int_S\big|\phi\big(x_n(s)\big)-\phi\big(x(s)\big)\big|{\rm
d}\mu(s)\nrightarrow 0$. Then there exist $\varepsilon_0>0$ and a
subsequence of $(x_n)_{n\in\NN}$, for convenience, still denoted by
$(x_n)_{n\in\NN}$, such that
\begin{align}
\int_S\big|\phi\big(x_n(s)\big)-\phi\big(x(s)\big)\big|{\rm
d}\mu(s)\geq\varepsilon_0,\quad \forall n\in\NN.\label{valic:6}
\end{align}
Since $x\in L^{\infty}_E(S,\mu)$, there exists $L>0$ such that $\|x(s)\|\leq L$ for almost all $s\in S$. We can and do suppose that \begin{align}
\|x(s)\|\leq L,\quad \forall s\in S.\label{valic:5}
\end{align}
Let $\varepsilon>0$ .
Since $\phi$ is  continuous, then $\phi$ is uniformly continuous on $(L+1)B_E$. Then  there exists $\delta>0$ such that
\begin{align}
|\phi(u)-\phi(v)|\leq\varepsilon,\quad\forall \|u-v\|\leq\delta, \forall u,v\in (L+1)B_E.\label{valic:7}
\end{align}
By Fact~\ref{vicon:2}, there exists  a subsequence $(x_{n_k})_{k\in\NN}$
of $(x_n)_{n\in\NN}$ such that $(x_{n_k})_{k\in\NN}$ converges to $x$ $\mu$--uniformly.
Then there exist $ N_1\in\NN$ and  a measurable subset $T$ of $S$ such that
$\mu(T)<\varepsilon$ and
\begin{align}
\|x_{n_k}(s)-x(s)\|\leq \min\{\delta,1\},\quad\forall k\geq N_1,\forall s\in T^c.\label{valic:8}
\end{align}
Then by \eqref{valic:5},
\begin{align}
x_{n_k}(s)\in( L+1)B_E,\quad\forall k\geq N_1,\forall s\in T^c.\label{valic:9}
\end{align}
Then by assumption, we have
\begin{align*}
&\int_S\big|\phi\big(x_{n_k}(s)\big)-\phi\big(x(s)\big)\big|{\rm d}\mu(s)\\
&=
\int_{T^c}\big|\phi\big(x_{n_k}(s)\big)-\phi\big(x(s)\big)\big|{\rm
d}\mu(s)
+\int_T\big|\phi\big(x_{n_k}(s)\big)-\phi\big(x(s)\big)\big|{\rm d}\mu(s)\\
&\leq \int_{T^c}\varepsilon {\rm
d}\mu(s)+\int_T\big|\phi\big(x_{n_k}(s)\big)-\phi\big(x(s)\big)\big|{\rm
d}\mu(s)\quad\text{(by \eqref{valic:9},
 \eqref{valic:5},\eqref{valic:7} and \eqref{valic:8})}\\
&\leq \int_{T^c}\varepsilon {\rm d}\mu(s)+\int_T 2M {\rm d}\mu(s)\quad\text{(since $|\phi(v)|\leq M$ for all $v\in E$)}\\
&\leq \varepsilon\mu(T^c)+ 2M\varepsilon,\quad \forall k\geq N_1.
\end{align*}
Then
$\int_S\big|\phi\big(x_{n_k}(s)\big)-\phi\big(x(s)\big)\big|{\rm
d}\mu(s)\longrightarrow 0$, which contradicts
 \eqref{valic:6}.  Hence
$\int_S\big|\phi\big(x_n(s)\big)-\phi\big(x(s)\big)\big|{\rm
d}\mu(s)\longrightarrow 0$.
 Consequently, $I_\phi (x_n)\longrightarrow I_\phi (x)$.

 The proof is similar when $\phi$ is assumed uniformly continuous but $x$ is allowed to  lie in $L^{1}_E(S,\mu)$.
\end{proof}

The next result replaces continuity conditions on $\phi$ by a boundedness requirement on the range of its Clarke subdifferential.

\begin{theorem}[\textbf{Preservation of convergence in measure, II}]\label{Prpvitc:3} Let $\phi:E\rightarrow\RX$ be proper lower semicontinuous. Let $(x_n)_{n\in\NN}$ and $x$ be in $\dom I_\phi$.
Assume that there exists $\delta>0$ such that
\begin{align}\sup_{(x,x^*)\in\gra\partial_C\phi}\|x^*\|\leq\delta.\label{delb}\end{align}
Suppose that $(x_n)_{n\in\NN}$ converges to $x$ in measure
and there exists $M>0$ such that $|\phi(v)|\leq M$ for all $v\in \dom \phi$.
 Then $\ I_\phi(x_n)\longrightarrow I_\phi(x)$.
\end{theorem}

\begin{proof}
Since $(x_n)_{n\in\NN}$ and $x$ are in $\dom I_\phi$, we can and do assume that $x_{n}(s)\in\dom\phi$ for all $n\in\NN, s\in S$ and $x(s)\in\dom \phi$ for all $s\in S$.

Let $\varepsilon>0$. Since $(x_n)_{n\in\NN}$ converges to $x$ in measure,
there exists $N_1\in\NN$ such that
\begin{align}
\mu(T_n)<\varepsilon,\quad \forall n\geq N_1,\label{valic:10a1}
\end{align}
\allowdisplaybreaks
where $T_n:=\big\{s\in S\mid \|x_{n}(s)-x(s)\|\geq\varepsilon\big\}$.
Then for every $ n\geq N_1$, by Fact~\ref{mvat:1},
there exists $y^*_{n_s}\in\partial_C f(y_{n_s})$ for all $s\in S$ such that
\begin{align*}
&\limsup \big(I_\phi(x_n)- I_\phi(x)\big){\rm d}\mu(s)\\
&=\limsup \int_S\Big(\phi\big(x_{n}(s)\big)-\phi\big(x(s)\big)\Big){\rm d}\mu(s)\\
&\leq\limsup \int_{T_n}
\Big(\phi\big(x_{n}(s)\big)-\phi\big(x(s)\big)\Big){\rm d}\mu(s)+
\limsup \int_{(T_n)^c} \Big(\phi\big(x_{n}(s)\big)-\phi\big(x(s)\big)\Big){\rm d}\mu(s)\\
&\leq \limsup \int_{T_n} 2M {\rm d}\mu(s)+
\limsup \int_{(T_n)^c} \Big(\phi\big(x_{n}(s)\big)-\phi\big(x(s)\big)\Big){\rm d}\mu(s)\\
&\leq\limsup 2M\mu(T_n)+
\limsup \int_{(T_n)^c} \big\langle y^*_{n_s}, x_{n}(s)-x(s)\big\rangle +\tfrac{1}{n}{\rm d}\mu(s)\\
&\leq 2M\varepsilon+
\limsup \int_{(T_n)^c} \|y^*_{n_s}\|\cdot \|x_{n}(s)-x(s)\| +\tfrac{1}{n}{\rm d}\mu(s)\quad \text{(by \eqref{valic:10a1})}\\
&\leq 2M\varepsilon+
\limsup \int_{(T_n)^c}\delta \varepsilon +\tfrac{1}{n}{\rm d}\mu(s)\quad \text{(by the assumption \eqref{delb})}\\
&\leq 2M\varepsilon+
\delta \varepsilon\mu(S).
\end{align*}
Then letting $\varepsilon\longrightarrow 0$ in the above equation, we have
$\limsup I_\phi(x_n)\leq I_\phi(x)$.

Similarly, we have
\begin{align*}
I_\phi(x)-\liminf I_\phi(x_n)=\limsup\big(I_\phi(x)-I_\phi(x_n))\leq0.
\end{align*}
Combining above results, we have
\begin{align*}
I_\phi(x)\leq\liminf I_\phi(x_n)\leq\limsup I_\phi(x_n)\leq I_\phi(x).
\end{align*}
Hence $I_\phi(x_n)\longrightarrow I_\phi(x)$.
\end{proof}

While convex integrands will not satisfy \eqref{delb} there are many simple examples which do.

\begin{example}[\textbf{Nonconvex integrands}]
Let $\phi(x):=\min\{\|x\|,1\}$ for every $x\in E$.
Let $(x_n)_{n\in\NN}$ and $x$ be in $L^1_E(S,\mu)$.
Suppose that  $(x_n)_{n\in\NN}$ converges to $x$ in measure.
 Then $I_\phi (x_n)\longrightarrow I_\phi (x)$.
\end{example}
\begin{proof}
Clearly, $\phi$ is continuous  (actually Lipschitz) and $\sup_{(x,x^*)\in\gra\partial_C\phi}\|x^*\|\leq 1$. By the definition of
$\phi$, we have $(x_n)_{n\in\NN}$ and $x$ are in $\dom I_\phi$.
Then directly apply Theorem~\ref{Prpvitc:3}.\end{proof}

To use such value convergence results, it behooves us to provide an
example of integrands such that
$\int_S\big|\phi\big(x_n(s)\big)-\phi\big(x(s)\big)\big|\,{\rm
d}\mu(s)\longrightarrow 0$ implies $x_n\rightarrow x$ in measure.

\begin{example}\label{Exlbr:1}
Let $\phi(x):=\begin{cases}
-\log x,
\,&\text{if}\,
x>0;\\
+\infty,&\text{otherwise}
\end{cases}, \forall x\in\RR$. Let $S=\left[0,1\right]$  and let $\mu$ be Lebesgue measure.
 Let $(x_n)_{n\in\NN}$ and $x$ be in $\dom I_{\phi}$. Suppose that $x\in L^{\infty}_E(S,\mu)$ and $\int_S\big|\phi\big(x_n(s)\big)-\phi\big(x(s)\big)\big|{\rm d}\mu(s)\longrightarrow 0$. Then $x_n\rightarrow x$ in measure.
\end{example}
\begin{proof}
By the assumption, we can and do assume that $x_{n}(s)\in\dom\phi$ for all $n\in\NN, s\in S$ and $x(s)\in\dom \phi$ for all $s\in S$. Since $x\in L^{\infty}_E(S,\mu)$, there exists $L>0$ such that $|x(s)|\leq L$ for almost all $s\in S$. We can and do suppose that \begin{align}
|x(s)|\leq L,\quad \forall s\in S.\label{valic:15}
\end{align}

Suppose to the contrary that $(x_n)_{n\in\NN}$ does not converge to
$x$ in measure. Then there exist $\eta>0$, $\varepsilon_0>0$ and a
subsequence $(x_{n_k})_{k\in\NN}$ of $(x_n)_{n\in\NN}$ such that
\begin{align}
\mu(T_k)\geq\varepsilon_0,\quad \forall k\in\NN,\label{valic:16}
\end{align}
where $T_k:=\big\{s\in S\mid \big|x_{n_k}(s)-x(s)\big|\geq\eta\big\}$.
Then we have
\begin{align*}
&\int_S\big|\phi\big(x_{n_k}(s)\big)-\phi\big(x(s)\big)\big|{\rm d}\mu(s)\\
&=\int_{T_k}
\big|\phi\big(x_{n_k}(s)\big)-\phi\big(x(s)\big)\big|{\rm d}\mu(s)+
\int_{(T_k)^c} \big|\phi\big(x_{n_k}(s)\big)-\phi\big(x(s)\big)\big|{\rm d}\mu(s)\\
&\geq \int_{T_k} \big|\phi\big(x_{n_k}(s)\big)-\phi\big(x(s)\big)\big|{\rm d}\mu(s)\\
&=\int_{T_k} \big|\big\langle \phi'(y_{k_s}), x_{n_k}(s)-x(s)\big\rangle\big|{\rm d}\mu(s),\quad \exists y_{k_s}\in \left[ x_{n_k}(s), x(s)\right]\quad\text{(by Mean Value Theorem)}\\
&=\int_{T_k} \tfrac{1}{\big|x(s)+t_{k_s}\big(x_{n_k}(s)-x(s)\big)\big|}\cdot \big| x_{n_k}(s)-x(s)\big|{\rm d}\mu(s),\quad \exists t_{k_s}\in \left[ 0, 1\right]\\
&\geq\int_{T_k} \tfrac{1}{\big|x(s)\big|+\big|x_{n_k}(s)-x(s)\big|}\cdot \big| x_{n_k}(s)-x(s)\big|{\rm d}\mu(s)\\
&\geq \int_{T_k} \tfrac{\eta}{L+\eta}\geq\tfrac{\eta\varepsilon_0}{L+\eta}\quad\text{(by \eqref{valic:15} and \eqref{valic:16})}, \forall k\in\NN,
\end{align*}
which contradicts that
$\int_S\big|\phi\big(x_n(s)\big)-\phi\big(x(s)\big)\big|{\rm
d}\mu(s)\longrightarrow 0$. Hence $x_n\rightarrow x$ in measure.
\end{proof}

Sadly, in Example~\ref{Exlbr:1}, we cannot replace $\int_S\big|\phi\big(x_n(s)\big)-\phi\big(x(s)\big)\big|{\rm d}\mu(s)\longrightarrow 0$ by
$I_\phi(x_n)\longrightarrow I_\phi(x)$.  We use the following example to show that.

\begin{example}\label{Exlbr:2}
Let $\phi(x):=\begin{cases}
-\log x,
\,&\text{if}\,
x>0;\\
+\infty,&\text{otherwise}
\end{cases}, \forall x\in\RR$, and let $S,\mu$ be defined as in Example~\ref{Exlbr:1}. We define $x_n: S\rightarrow\RR$
(for every $n\in\NN$) by
$x_n(s):=\begin{cases}
n,
\,&\text{if}\,
s\in\left[0,\tfrac{1}{1+\log n}\right];\\
1,&\text{otherwise}
\end{cases}, \forall s\in S$.
Set $x(s):=\exp(1), \forall s\in S$.

Then $(x_n)_{n\in\NN}$ and $x$ are in
$\dom I_\phi$, $x\in L_\RR^{\infty}(S,\mu)$ and $I_\phi(x_n)\longrightarrow I_\phi(x)=-1$ but $(x_n)_{n\in\NN}$ does not converge to $x$ in measure.
\end{example}

\begin{proof}
Clearly, $x\in L^1_\RR(S,\mu)\cap  L^{\infty}_\RR(S,\mu)$. Now we show $(x_n)_{n\in\NN}$ is in
$L^1_\RR(S,\mu)$.  Fix $n\in\NN$. Then $x_n$ is a measurable function. We have
\begin{align*}
\int_S |x_n(s)| {\rm d}\mu(s)&=\int_S x_n(s) {\rm d}\mu(s)=
\int_{\left[0,\tfrac{1}{1+\log n}\right]} x_n(s) {\rm d}\mu(s)
+\int_{\left]\tfrac{1}{1+\log n}, 1\right]} x_n(s) {\rm d}\mu(s)\nonumber\\
&=\int_{\left[0,\tfrac{1}{1+\log n}\right]} n {\rm d}\mu(s)
+\int_{\left]\tfrac{1}{1+\log n}, 1\right]} 1{\rm d}\mu(s)\nonumber\\
&\leq\tfrac{n}{1+\log n}+1<+\infty\nonumber.
\end{align*}
Thus, $x_n\in L^1_\RR(S,\mu)$.

Now we show that $I_\phi(x_n)\longrightarrow I_\phi(x)$.
Clearly, $I_\phi(x)=\int_S -\log\Big(\exp(1)\big){\rm d}\mu(s)=-1$.
\begin{align*}
I_\phi(x_n)&=\int_S \phi\big(x_n(s)\big){\rm d}\mu(s)\\
&=\int_{\left[0,\tfrac{1}{1+\log n}\right]} \phi\big(x_n(s)\big){\rm
d}\mu(s)
+\int_{\left]\tfrac{1}{1+\log n},1\right]} \phi\big(x_n(s)\big){\rm d}\mu(s)\\
&=\int_{\left[0,\tfrac{1}{1+\log n}\right]} -\log n  {\rm d}\mu(s)
+\int_{\left]\tfrac{1}{1+\log n},1\right]} -\log 1 {\rm d}\mu(s)\\
&=-\tfrac{\log n}{1+\log n}\longrightarrow -1=I_\phi(x).
\end{align*}
Hence $I_\phi(x_n)\longrightarrow I_\phi(x)$.

On the other hand,
\begin{align*}
&\mu\big\{s\in S\mid |x_n(s)-x(s)|\geq 1\big\}=
\mu\big\{s\in S\mid |x_n(s)-\exp(1)|\geq1\big\}\\
&\geq \mu\big\{\left]\tfrac{1}{1+\log n}, 1\right]\big\}
=1-\tfrac{1}{1+\log n}\geq\tfrac{1}{2},\quad \forall n\geq3.
\end{align*}
Hence
$(x_n)_{n\in\NN}$ does not converge to $x$ in measure.
\end{proof}

The converse of Example~\ref{Exlbr:1} cannot hold either.

\begin{example}\label{Exlbr:3}
Let $\phi, S,\mu$ and $(x_n)_{n\in\NN}$ be all defined as in Example~\ref{Exlbr:2}.
Let $x(s):=1, \forall s\in S$. Then $(x_n)_{n\in\NN}$ and $x$ are in
$\dom I_\phi$, $x\in L_\RR^{\infty}(S,\mu)$ and $x_n\rightarrow x$ in measure
but $\int_S\big|\phi\big(x_n(s)\big)-\phi\big(x(s)\big)\big|{\rm
d}\mu(s)\nrightarrow 0$.
\end{example}
\begin{proof}
Example~\ref{Exlbr:2} shows that $(x_n)_{n\in\NN}$ is in $\dom I_\phi$.

Clearly, $x\in\dom I_\phi$ and $x\in L_\RR^{\infty}(S,\mu)$. Now we show that $x_n\rightarrow x$ in measure.
Let $\eta>0$. Then we have
\begin{align*}
\mu\big\{s\in S\mid |x_n(s)-x(s)|\geq\eta\big\}&=
\mu\big\{s\in S\mid |x_n(s)-1|\geq\eta\big\}
\leq \mu\big\{\left[0,\tfrac{1}{1+\log n}\right]\big\}
=\tfrac{1}{1+\log n}\longrightarrow 0.
\end{align*}
Hence $x_n\rightarrow x$ in measure.

We have
\begin{align*}
&\lim\int_S\big|\phi\big(x_n(s)\big)-\phi\big(x(s)\big)\big|{\rm
d}\mu(s)
=\lim\int_S\big|\phi\big(x_n(s)\big)\big)\big|{\rm
d}\mu(s)\\
&=\lim\int_S-\phi\big(x_n(s)\big)\big){\rm
d}\mu(s)=\lim-I_\phi(x_n)=1\neq0\quad\text{( by Example~\ref{Exlbr:2})}.
\end{align*}
Hence $\int_S\big|\phi\big(x_n(s)\big)-\phi\big(x(s)\big)\big|{\rm
d}\mu(s)\nrightarrow 0$.
\end{proof}

Let $\phi:E\rightarrow\RX$ be proper lower semicontinuous and strictly convex. Let $(x_n)_{n\in\NN}$ and $x$ be in $\dom I_\phi$. Assume that $ x\in\argmin I_\phi$.
The results so far given provoke the following question:
\begin{quote}
If $x_n\longrightarrow x$ in measure,  is it necessarily true that
$I_\phi (x_n)\longrightarrow I_\phi (x)$?
\end{quote}

The following example shows that  above statement cannot be true without imposing extra conditions.

\begin{example}[\textbf{Incompatibility of measure and value convergence}]
Let
\\
 $\phi(x):=\begin{cases}
-\log x  +x,
\,&\text{if}\,
x>0;\\
+\infty,&\text{otherwise}
\end{cases}, \forall x\in\RR$. Let $S=\left[0,1\right]$  and  let $\mu$ be Lebesgue measure.
Set (for every $n\in\NN$)
$x_n(s):=\begin{cases}
n,
\,&\text{if}\,
s\in\left[0,\tfrac{1}{n}\right];\\
1,&\text{otherwise}
\end{cases}, \forall s\in S$.  Then $\phi$ is proper lower semicontinuous and strictly
convex.
Let $x:S\rightarrow\RR$  be given by $x(s):=1,\forall s\in S$.

Then $(x_n)_{n\in\NN}$ and  $x$ are in $L^1_\RR(S,\mu)$, \[\argmin I_\phi=\{x\}~\mbox{ and } 
x_n\rightarrow x~\mbox{in measure}\] but $I_\phi(x_n)\nrightarrow I_\phi(x)$. In particular,
$(x_n)_{n\in\NN}$ does not converge weakly to $x$.
\end{example}

\begin{proof}
Clearly, $x\in L^1_\RR(S,\mu)$. First we show $(x_n)_{n\in\NN}$ is in
$L^1_\RR(S,\mu)$.  Let $n\in\NN$. Then $x_n$ is a measurable function. We have
\begin{align}
\int_S |x_n(s)| {\rm d}\mu(s)&=\int_S x_n(s) {\rm d}\mu(s)=
\int_{\left[0,\tfrac{1}{n}\right]} x_n(s) {\rm d}\mu(s)
+\int_{\left]\tfrac{1}{n}, 1\right]} x_n(s) {\rm d}\mu(s)\nonumber\\
&=\int_{\left[0,\tfrac{1}{n}\right]} n {\rm d}\mu(s)
+\int_{\left]\tfrac{1}{n}, 1\right]} 1{\rm d}\mu(s)\nonumber\\
&=1+(1-\tfrac{1}{n})\leq1+1=2.\label{normwke:1}
\end{align}
Thus, $x_n\in L^1_\RR(S,\mu)$.

Since $\argmin\phi=\{1\}$, $I_\phi(x)=\int_S \phi(1){\rm
d}\mu(s)\leq \int_S \phi\big(z(s)\big){\rm d}\mu(s)
=I_\phi(z),\forall z\in L^1_\RR(S,\mu)$. Then $x\in\argmin I_\phi$.
By Fact~\ref{legenFa:2}, $I_\phi$  has unique minimizer and hence
$\argmin I_\phi=\{x\}$.

Now we show that $x_n\rightarrow x$ in measure.
Let $\eta>0$. Then we have
\begin{align*}
\mu\big\{s\in S\mid |x_n(s)-x(s)|\geq\eta\big\}&=
\mu\big\{s\in S\mid |x_n(s)-1|\geq\eta\big\}
\leq \mu\big\{\left[0,\tfrac{1}{n}\right]\big\}
=\tfrac{1}{n}\longrightarrow 0.
\end{align*}
Hence
$\lim\mu\big\{s\in S\mid |x_n(s)-x(s)|\geq\eta\big\}=0$ and thus
$x_n\rightarrow x$ in measure.

By \eqref{normwke:1}, $\|x_n\|_1\nrightarrow\|x\|_1=1$. Then $x_n\nrightarrow x$. Since $x_n\rightarrow x$ in measure, \cite[Lemma~2]{Vin} implies that $(x_n)_{n\in\NN}$ does not converge weakly to $x$.

We claim that $I_\phi(x_n)\nrightarrow I_\phi(x)$. We have
\begin{align*}
I_\phi(x_n)&=\int_S \phi\big(x_n(s)\big){\rm d}\mu(s)\\
&=\int_{\left[0,\tfrac{1}{n}\right]} \phi\big(x_n(s)\big){\rm
d}\mu(s)
+\int_{\left]\tfrac{1}{n},1\right]} \phi\big(x_n(s)\big){\rm d}\mu(s)\\
&=\int_{\left[0,\tfrac{1}{n}\right]} -\log n +n {\rm d}\mu(s)
+\int_{\left]\tfrac{1}{n},1\right]} -\log 1  +1{\rm d}\mu(s)\\
&=-\tfrac{\log n}{n}+1 +(1-\tfrac{1}{n})\longrightarrow 2.
\end{align*}
However,
\begin{align*}
I_\phi(x)&=\int_S \phi\big(x(s)\big) {\rm d}\mu(s)= \int_S -\log 1
+1 {\rm d}\mu(s)=1.
\end{align*}
Combining the results above, $I_\phi(x_n)\nrightarrow I_\phi(x)$.
\end{proof}

\begin{remark}
Let $(C_n)_{n\in\NN}$ and $C_{\infty}$ in $L^{1}_E(S,\mu)$ be  closed convex sets, and let  $\phi:E\rightarrow\RX$ be  proper lower semicontinuous and convex.  When, as in  \cite{BorLew1}, we consider the following sequences of optimization problems
\begin{align*}
(P_n)&\qquad \qquad\qquad\qquad V(P_n):=\inf\big\{I_\phi(x)\mid x\in C_n\big\},\\
(P_{\infty})&\qquad\qquad \qquad\qquad V(P_{\infty}):=\inf\big\{I_\phi(x)\mid x\in C_{\infty}\big\},
\end{align*}
the above results indicate that one cannot significantly weaken the conditions of Fact~\ref{BLal:4a} (such as, replacing weak convergence by measure convergence).
\end{remark}

To conclude, we observe that the examples of this section indicate the limited use of convergence in measure in the absence of weak compactness conditions.

\section*{Acknowledgments} The authors thank Dr.~Jon Vanderwerff for helpful discussions.
Jonathan  Borwein and Liangjin Yao both were  partially supported by various Australian Research  Council grants.


\end{document}